\documentclass{siamart190516}

\usepackage[utf8]{inputenc}
\usepackage[T1]{fontenc}

\usepackage{amsmath}
\usepackage{amssymb}
\usepackage{enumerate}
\usepackage{dsfont}
\usepackage{xcolor}
\usepackage{colortbl}
\definecolor{hage}{rgb}{0.4,0.6,1}

\usepackage{tikz}
\usepackage{pgfplots}
\usepackage{caption}
\usepackage{subcaption}
\usetikzlibrary{fit}
\usetikzlibrary{calc}
\pgfdeclarelayer{background}
\pgfsetlayers{background,main}
\colorlet{inbox}{lightgray!20}
\colorlet{outbox}{lightgray!50}

\makeatletter
\renewcommand*\env@matrix[1][*\c@MaxMatrixCols c]{%
  \hskip -\arraycolsep
  \let\@ifnextchar\new@ifnextchar
  \array{#1}}

\newcommand{\cB}{\mathcal{B}}

\newcommand{\cA}{\mathcal{A}}

\newcommand{\cD}{\mathcal{D}}

\newcommand{\cF}{\mathcal{F}}

\usepackage{mathdots}

\newtheorem{remark}[theorem]{Remark}

\newcommand{\RN}[1]{\uppercase\expandafter{\romannumeral#1}}
\newcommand{\eps}{\varepsilon}

\newcommand{\N}{{\mathbb{N}}}

\newcommand{\R}{{\mathbb{R}}}

\newcommand{\fB}{\mathfrak{B}}
\newcommand{\fa}{\mathfrak{a}}
\newcommand{\fH}{\mathfrak{H}}
\newcommand{\fV}{\mathfrak{V}}
\newcommand{\fA}{\mathfrak{A}}

\DeclareMathOperator{\esssup}{ess\, sup}

\let\div\relax 
\DeclareMathOperator{\div}{div} 

\newcommand{\setdef}[2]{\left\{\, #1 \left|\, \vphantom{#1} #2\right.\right\}}
\newcommand{\ddt}{\tfrac{\text{\normalfont d}}{\text{\normalfont d}t}}

\newcommand{\diag}{\text{\rm diag\,}}
\newcommand{\ds}[1]{{\rm \, d} #1 \,}

\DeclareMathOperator{\loc}{loc}

{\left(\begin{smallmatrix}}
{\end{smallmatrix}\right)}

{\left[\begin{smallmatrix}}
{\end{smallmatrix}\right]}

\renewcommand{\theenumi}{{\rm(\roman{enumi})}} 

\title{Funnel control of the Fokker-Planck equation for a multi-dimensional Ornstein-Uhlenbeck process\thanks{Preprint submitted to SIAM Journal on Control and Optimization, \today\funding{This work was supported by the German Research Foundation (Deutsche Forschungsgemeinschaft) via the grant BE\,6263/1-1.}}}

\author{Thomas Berger\thanks{Institut f\"ur Mathematik, Universit\"at Paderborn, Warburger Str.~100, 33098~Paderborn, Germany
  (\email{thomas.berger@math.upb.de}).}}

\begin{document}

\maketitle

\begin{abstract}
In this paper the feasibility of funnel control techniques for the Fokker-Planck equation corresponding to a multi-dimensional Ornstein-Uhlenbeck process on an unbounded spatial domain is explored. First, using weighted Lebesgue and Sobolev spaces, an auxiliary operator is defined via a suitable sesquilinear form. This operator is then transformed to the desired Fokker-Planck operator. We show that any mild solution of the controlled Fokker-Planck equation (which is a probability density) has a covariance matrix that exponentially converges to a constant matrix. After a simple feedforward control approach is discussed, we show feasibility of funnel control in the presence of disturbances by exploiting semigroup theory. We emphasize that the closed-loop system is a nonlinear and time-varying PDE. The results are illustrated by some simulations.
\end{abstract}



\section{Introduction}

In this work we study output tracking control for the Fokker-Planck equation that corresponds to a multi-dimensional Ornstein-Uhlenbeck process. The latter is a continuous-time stochastic process which was originally used to describe the motion of a massive Brownian particle under the influence of friction~\cite{UhleOrns30}. Although its investigation was mainly driven by physics and mathematics, several other important applications emerged, such as in neurobiology~\cite{RiccSace79} and in finance~\cite{SchoZhu99}. The Ornstein-Uhlenbeck process (typically in the one-dimensional case) is often considered in the context of optimal control, see e.g.~\cite{AnnuBorz10,AnnuBorz13,FleiGrun16,FleiGugl16}. The Fokker-Planck equation is a parabolic partial differential equation (PDE) which describes the evolution of the probability density function of the solution of a stochastic differential equation, see e.g.~\cite{Okse03}. It will be the main tool to treat the output tracking control problem.

In this context, control means that we assume that the drift term of the stochastic differential equation can be manipulated by an external signal, which is called the control input and enters the equation via a nonlinear function~$g$ satisfying a so-called high-gain property. The resulting Fokker-Planck equation can be viewed as an abstract bilinear control system in terms of the state and the (nonlinear) control function, cf.~\cite{BreiKuni18,HosfJaco20}; see also the monograph~\cite{Khap10} for several topics on bilinear control systems. The mean value (or expected value) of the Ornstein-Uhlenbeck process is chosen as the output~$y$ and measurements of it are assumed to be available. For a given reference signal~$y_{\rm ref}$, we then seek to achieve that the (norm of the) difference between the mean value and the reference stays within a prescribed error margin (given by a function~$\varphi$) for all times, thus allowing to control the mean value of the process as desired. Under funnel control the closed-loop system is a nonlinear and time-varying PDE of the form
\begin{equation}\label{eq:funnel-CL}
\begin{aligned}
 \dot p(t,x) &= \div \left( c \nabla p(t,x) + p(t,x) \Big(\Gamma x - g\left( \big(N\circ \alpha\big)\big(\|w(t)\|_{\R^n}^2\big) w(t)\right)\Big)\right) + d(t,x),\\
      w(t) &= \varphi(t) \big(y(t) - y_{\rm ref}(t)\big),\quad   y(t) = \int_{\R^n} x p(t,x)\, {\rm d}x,
\end{aligned}
\end{equation}
with the initial condition $p(0,x) = p_0(x)$, for which we prove existence and uniqueness of bounded global solutions. In the above equation, {$c>0$ and $\Gamma\in\R^{n\times n}$ are diffusion and drift coefficients, resp.,} $d$ is a bounded disturbance and the funnel control input is given by
\[
    u(t) = \big(N\circ \alpha\big)\big(\|w(t)\|_{\R^n}^2\big) w(t),
\]
where $\alpha:[0,1)\to [1,\infty)$ is a bijection and $N$ is a switching function.

We like to stress that we do not require knowledge of the system parameters or the initial probability density. {This is different from other approaches as e.g.~\cite{ChenGeor16}, where the probability density is steered to a desired density function, but the initial density must be known.}

Furthermore, {by controlling the mean value of the process we may indeed influence the behavior of the} entire probability density function. Since only the drift term in the Fokker-Planck equation is influenced by the control input, the covariance matrix of the process is independent of it. We will show that it converges exponentially to $c\Gamma^{-1}$. Indeed, simulations show that the shape of the probability density does not change after some initial time, and is essentially only shifted according to the movement of the mean value.

The control law for the input~$u$ is based on the funnel control methodology developed in~\cite{IlchRyan02b}. The funnel controller is an output-error feedback of high-gain type. Its advantages are that it is model-free (i.e., it requires no knowledge of the system parameters or the initial value), it is robust and of striking simplicity -- for the Fokker-Planck equation we will show that robustness can be guaranteed  w.r.t.\ additive disturbances ``with zero mass''. The funnel controller has been successfully applied e.g.\ in temperature control of chemical reactor models~\cite{IlchTren04}, control of industrial servo-systems~\cite{Hack17} and underactuated multibody systems~\cite{BergOtto19}, voltage and current control of electrical circuits~\cite{BergReis14a}, DC-link power flow control~\cite{SenfPaug14} and adaptive cruise control~\cite{BergRaue18,BergRaue20}.

Funnel control for infinite-dimensional systems is a hard task in general. A simple class of systems with relative degree one and infinite-dimensional internal dynamics has been considered in the seminal work~\cite{IlchRyan02b}. Linear infinite-dimensional systems for which an integer-valued relative degree exists have been considered in~\cite{IlchSeli16}. In fact, it has been observed in the recent work~\cite{BergPuch20a} that the existence of an integer-valued relative degree is essential to apply known funnel control results as formulated e.g.\ in~\cite{BergLe18a}. It is then shown in~\cite{BergPuch20a} that a large class of systems which exhibit infinite-dimensional internal dynamics is susceptible to funnel control. A practically relevant example is a mowing water tank system, which is shown to belong to the aforementioned class in~\cite{BergPuch19}. However, not even every linear infinite-dimensional system has a well-defined relative degree, in which case the results from~\cite{BergLe18a} cannot be applied. For this class of systems -- to which the Fokker-Planck equation belongs -- the feasibility of funnel control has to be investigated directly for the (nonlinear and time-varying) closed-loop system; see e.g.~\cite{ReisSeli15b} for a boundary controlled heat equation,~\cite{PuchReis19pp} for a general class of boundary control systems and~\cite{BergBrei21} for a system of monodomain equations (which represent defibrillation processes of the human heart).

\subsection{Nomenclature}

The set of natural numbers is denoted by~$\N$ and $\N_0 = \N \cup\{0\}$. 
For a measurable set $\Omega\subseteq\R^n$, $n\in\N$, a measurable function $w:\Omega\to\R_{\ge 0}$ and $p\in[1,\infty]$, $L^p(\Omega;w)$ denotes the $w$-weighted Lebesgue space of (equivalence classes of) measurable and $p$-integrable functions $f:\Omega\to\R$ with norm
\[
    \|f\|_{L^p(\Omega;w)} = \left( \int_\Omega w(x)\, |f(x)|^p {\rm d}x \right)^{1/p},\quad f\in L^p(\Omega;w),
\]
if $p<\infty$ and $\|f\|_{L^\infty(\Omega;w)} = \esssup_{x\in\Omega} w(x)\, |f(x)|$ if $p=\infty$. Additionally, for $k\in\N_0$, $W^{k,p}(\Omega;w)$ denotes the $w$-weighted Sobolev space of (equivalence classes of) $k$-times weakly differentiable functions $f:\Omega\to\R$ with $f, f', \ldots, f^{(k)} \in L^p(\Omega;w)$. If $w\equiv 1$, then we write $L^p(\Omega;1) = L^p(\Omega)$ and $W^{k,p}(\Omega;1) = W^{k,p}(\Omega)$. 
%
The space $L^2(\Omega;w)^n$ is equipped with the inner product
\[
    \langle f_1, f_2\rangle_{L^2(\Omega;w)^n} = \sum_{k=1}^n \langle f_{1,k}, f_{2,k}\rangle_{L^2(\Omega;w)}.
\]
For an interval $J\subseteq\R$, a Banach space $X$ and $p\in[1,\infty]$, we denote by $L^p(J;X)$ the vector space of equivalence classes of strongly measurable functions $f:J\to X$ such that $\|f(\cdot)\|_X\in L^p(J)$; the distinction between $L^p(J;X)$ and $L^p(\Omega;w)$ should be clear from the context.
If $J=(a,b)$ for $a,b\in\R$, we simply write $L^p(a,b;X)$, also for the case $a=-\infty$ or $b=\infty$. We refer to~\cite{Adam75} for further details on Sobolev and Lebesgue spaces.

By $C^k(J;X)$ we denote the space of $k$-times continuously differentiable functions $f:J\to X$, $k\in\N_0$, with $C(J;X) := C^0(J;X)$.
For $p\in[1,\infty]$, $W^{1,p}(J;X)$ stands for the Sobolev space of $X$-valued equivalance classes of weakly differentiable and $p$-integrable functions $f:J\to X$ with $p$-integrable weak derivative, i.e., $f,\dot{f}\in L^p(J;X)$. Thereby, integration (and thus weak differentiation) has to be understood in the Bochner sense, see~\cite[Sec.~5.9.2]{Evan10}.
The spaces $L^p_{\rm loc}(J;X)$ and $W^{1,p}_{\rm loc}(J;X)$ consist of all~$f$ whose restriction to any compact interval $K\subseteq J$ are in $L^p(K;X)$ or $W^{1,p}(K;X)$, respectively.


By $\cB(X;Y)$, where $X, Y$ are Hilbert spaces, we denote the set of all bounded linear operators $A:X\to Y$. Recall that a {\it $C_0$-semigroup}  $(T(t))_{t\ge0}$ on $X$ is a $\mathcal{L}(X;X)$-valued map satisfying $T(0)=I_{X}$ and $T(t+s)=T(t)T(s)$, $s,t\geq0$, where $I_{X}$ denotes the identity operator, and $t\mapsto T(t)x$ is continuous for every $x\in X$. $C_0$-semigroups are characterized by their generator~$A$, which is a, not necessarily bounded, operator on~$X$. If $\|T(t)\|_{\cB(X;X)} \le 1$ for all $t\ge 0$, then $(T(t))_{t\ge0}$ is called a \textit{contraction semigroup}. For the notion of an \textit{analytic semigroup} (sometimes called holomorphic semigroup) we refer to~\cite[Sec.~3.10]{Staf05}.

Furthermore, recall the space $X_{-1}$, see e.g.~\cite[Sec.~2.10]{TucsWeis09}, which should be thought of as an abstract Sobolev space with negative index. If $A:\cD(A)\subseteq X\to X$ is a densely defined operator with $\rho(A)\neq \emptyset$, where $\rho(A)$ denotes the resolvent set of~$A$, then for any $\beta\in\rho(A)$ we denote by $X_{-1}$ the completion of $X$ with respect to the norm
\[
    \|x\|_{X_{-1}}=\|(\beta I-A)^{-1}x\|_X,\quad x\in X.
\]
The norms generated as above for different $\beta\in\rho(A)$ are equivalent and, in particular, $X_{-1}$ is independent of the choice of~$\beta$. If $A$ generates a $C_0$-semigroup $(T(t))_{t\ge 0}$ in $X$, then the latter has a unique extension to a semigroup $(T_{-1}(t))_{t\ge 0}$ in~$X_{-1}$, which is given by
\[
    T_{-1}(t) = (\beta I - A_{-1}) T(t),\quad t\ge 0,
\]
where $(\beta I - A_{-1}) \in \cB(X;X_{-1})$ is a surjective isometry. Therefore, $A_{-1}$ is the generator of the semigroup $(T_{-1}(t))_{t\ge 0}$.

In infinite-dimensional linear systems theory with unbounded control operators, the existence of mild solutions is closely related to the notion of \textit{admissibility}, see e.g.~\cite{TucsWeis09}. Let $U, X, Y$ be Hilbert spaces and $A$ as above such that it generates a $C_0$-semigroup $(T(t))_{t\ge 0}$ on $X$. Then we recall that $B\in \cB(U;X_{-1})$ is a {\it $L^p$-admissible}  control operator (for $(T(t))_{t\ge 0}$), with $p\in [1,\infty]$, if
\[
    \forall\,  t\ge 0\ \forall\, u\in L^{p}([0,t];U):\ \int_{0}^{t} {T}_{-1}(t-s)Bu(s)\ds{s} \in X.
\]

\subsection{The Fokker-Planck equation for a controlled stochastic process}

We consider a controlled stochastic process described by the It\^{o} stochastic differential equation (cf.~\cite[Sec.~11]{Okse03})
\begin{equation}\label{eq:SDE}
     {\rm d} X_t = b(t,X_t,u(t)) {\rm d} t + {\sigma(t,X_t)}  {\rm d} W_t,\quad X(t=0) = X_0,
\end{equation}
where $X_t:\Omega\to\R^n$, $t\ge 0$, are random vectors and $\Omega$ is the sample space of a probability space $(\Omega,\cF,P)$. $(W_t)_{t\ge0}$ denotes a $d$-dimensional Wiener process with zero mean value and unit variance, $b:\R_{\ge 0} \times \R^n \times \R^m\to \R^n$ is a drift function and {$\sigma:\R_{\ge 0} \times \R^n \to \R^{n\times d}$ is a diffusion coefficient.} 
The function $u:\R_{\ge 0}\to\R^m$ is the control input.

Using the framework presented in~\cite{AnnuBorz10} we can formulate the control problem for the probability density function of the stochastic process~$(X_t)_{t\ge 0}$ as a partial differential equation, the Fokker-Planck equation. This approach is feasible under appropriate assumptions on the functions~$b$ and~$\sigma$ as shown in~\cite{PrimKont04,Prot05}.  Define
\[
    {C:\R_{\ge 0} \times \R^n \to \R^{n\times n},\ (t,x) \mapsto \tfrac12 \sigma(t,x) \sigma(t,x)^\top,}
\]
then the probability density function $p:\R_{\ge0}\times \R^n\to\R$ associated with the process $(X_t)_{t\ge 0}$ evolves according to the Fokker-Planck equation
\begin{equation}\label{eq:FPE}
\begin{aligned}
  \frac{\partial p}{\partial t}(t,x) &= -\sum_{i=1}^n \frac{\partial}{\partial x_i} \big( b_i(t,x,u(t)) p(t,x)\big) && \\
   &\quad + \sum_{i,j=1}^n \frac{\partial^2}{\partial x_i \partial x_j} \big({C_{ij}(t,x)} p(t,x)\big), &&\text{in}\ (0,\infty)\times \R^n,\\
  p(0,x) &= p_0(x), &&\text{in}\ \R^n, 
\end{aligned}
\end{equation}
and additionally, since~$p$ is a probability density, we require
\begin{equation}\label{eq:FPE-cond}
\begin{aligned}
  p(t,x) &\ge 0,\quad  \text{in}\ [0,\infty)\times \R^n,\\
  \int_{\R^n} p(t,x) {\rm d}x &= 1,\quad  \text{in}\ [0,\infty).
\end{aligned}
\end{equation}

The second condition in~\eqref{eq:FPE-cond} is the conservation of probability, while the first requires any probability to be non-negative.
Some conditions for the existence of nonnegative solutions of the Fokker-Planck equation are given in~\cite{AnnuBorz13,BreiKuni18,FleiGugl16} for instance.

\subsection{The Ornstein-Uhlenbeck process}

As a specific stochastic process, in this work we consider a multi-dimensional Ornstein-Uhlenbeck process and we assume that it can be controlled via the drift term only. Then it is modelled by an equation of the form~\eqref{eq:SDE} with $m=n\in\N$, $d\in\N$ and
\[
    b(t,x,u) = g(u) - \Gamma x,\quad {\sigma(t,x)} = S \in\R^{n\times d},\quad \Gamma\in\R^{n\times n}.
\]
A special one-dimensional version of this with $n = d = 1$, $g(u)=u$ and $\Gamma, S > 0$ is often encountered in the literature, see e.g.~\cite{AnnuBorz10,AnnuBorz13,FleiGrun16} and the references therein. Let us further stress that the equation is restricted to a bounded spatial domain in many works such as~\cite{AnnuBorz10,AnnuBorz13}, and Dirichlet boundary conditions are used; this is not the natural framework, cf.\ also Section~\ref{Sec:FPO}.

In the present work we assume that
\begin{enumerate}
  \item $C:= \tfrac12 S S^\top = c I_n$ for some $c >0$,
  \item $\Gamma$ is symmetric and positive definite, written $\Gamma = \Gamma^\top > 0$,
\end{enumerate}
and the function $g\in C^1(\R^n;\R^n)$ is linearly bounded and satisfies a \emph{high-gain property}, i.e.,
\begin{equation}\label{eq:HGP}
\begin{aligned}
    \exists\, \bar g>0\ \forall\, v\in\R^n:&\ \|g(v)\|_{\R^n} \le \bar g \|v\|_{\R^n},\\
    \exists\, \delta\in (0,1):&\ \sup_{s\in\R} \min_{\delta\le \|v\|_{\R^n}\le 1} v^\top g(-sv) = \infty.
\end{aligned}
\end{equation}
Assumptions (i) and (ii) guarantee that the Fokker-Planck operator is self-adjoint and positive and its eigenfunctions can be computed explicitly. Assumption~\eqref{eq:HGP} is required for feasibility of the proposed funnel control method. The associated Fokker-Planck equation~\eqref{eq:FPE} is then given in the form
\begin{equation}\label{eq:FPE-OU}
\begin{aligned}
 \dot p(t,x) &= \div \big( c \nabla p(t,x) + p(t,x) \big(\Gamma x - g(u(t))\big)\big), &&\text{in}\ (0,\infty)\times \R^n,\\
  p(0,x) &= p_0(x), &&\text{in}\ \R^n,
\end{aligned}
\end{equation}
where $\dot p = \frac{\partial p}{\partial t}$. For later use we define the function
\begin{equation}\label{eq:phi}
    \phi:\R^n\to\R,\ x\mapsto \frac{1}{2c} x^\top \Gamma x.
\end{equation}
Since it is unrealistic to assume that we can measure~$p(t,x)$ for all $t\ge 0$ and all $x\in \R^n$, we associate an output function $y:\R_{\ge 0}\to \R^n$ with~\eqref{eq:FPE-OU}. The output should be chosen in such a way that, by manipulating it via the control input, it is possible to influence the collective behavior of the process. As mentioned in~\cite{AnnuBorz10}, the mean value $E[X_t]$ ``is omnipresent in almost all stochastic optimal control problems considered in the scientific literature''.  Therefore, it is a reasonable choice for the output, i.e.,
\begin{equation}\label{eq:output}
  y(t) = E[X_t] = \begin{pmatrix} E[X_{t,1}] \\ \vdots\\ E[X_{t,n}] \end{pmatrix} = \begin{pmatrix} \int_{\R^n} x_1\, p(t,x) {\rm d}x \\ \vdots\\ \int_{\R^n} x_n\, p(t,x) {\rm d}x \end{pmatrix}.
\end{equation}
We assume that the measurement of the output~$y(t)$ is available to the controller at each time $t\ge 0$. In practice, the corresponding integrals cannot be calculated exactly, thus the components of the mean value will typically be approximated by data-driven methods such as Monte Carlo integration.

Note that controlling the Fokker-Planck equation via the drift term with mean value as output is indeed sufficient to influence the {behavior} of the solution density, since the covariance matrix of the process is independent of the control input. In particular, provided~\eqref{eq:FPE-cond} holds, we will show in Proposition~\ref{Prop:sln-prop} that the covariance matrix of the solution satisfies
\[
   \lim_{t\to\infty} \int_{\R^n} \big(x-y(t)\big) \big(x-y(t)\big)^\top p(t,x){\rm d} x = c \Gamma^{-1}.
\]

\subsection{Control objective}\label{Ssec:ContrObj}

The objective is to design a robust output error feedback $u(t) = F(t,e(t))$, where $e(t)= y(t) - y_{\rm ref}(t)$ for some reference trajectory $y_{\rm ref}\in W^{1,\infty}(\R_{\ge 0};\R^n)$, such that in the closed-loop system the tracking error $e(t)$ evolves within a prescribed performance funnel
\begin{equation}
\mathcal{F}_{\varphi} := \setdef{(t,e)\in\R_{\ge 0} \times\R^n}{\varphi(t) \|e\|_{\R^n} < 1},\label{eq:perf_funnel}
\end{equation}
which is determined by a function~$\varphi$ belonging to
\begin{equation}
\Phi \!:=\!
\left\{
\varphi\!\in\! C^1(\R_{\ge 0};\R)
\left|\!
\begin{array}{l}
\text{$\varphi(t)>0$ for all $t>0$},\ \liminf_{t\to\infty} \varphi(t)>0,\\
\exists\, \xi>0\ \forall\, t\ge0:\ |\dot \varphi(t)|\le \xi (1+\varphi(t))
\end{array}
\right.\!\!\!
\right\}.
\label{eq:Phir}\end{equation}
The robustness requirement on the control essentially means that it is feasible under bounded additive disturbances ``with zero mass'', which influence the Fokker-Planck equation. This is made precise in Section~\ref{Sec:SolProp}.

The performance funnel~$\mathcal{F}_{\varphi}$ accounts for the two objectives of~$y$ approaching~$y_{\rm ref}$ with prescribed transient behavior and asymptotic accuracy. Its boundary is given by the reciprocal of~$\varphi$, see Fig.~\ref{Fig:funnel}. We explicitly allow for $\varphi(0)=0$, meaning that no restriction on the initial value is imposed since $\varphi(0) \|e(0)\|_{\R^n} < 1$; the funnel boundary $1/\varphi$ has a pole at $t=0$ in this case. Furthermore, $\varphi$ may be unbounded and in this case asymptotic tracking may be achieved, i.e., $\lim_{t\to\infty} e(t) = 0$.

\begin{figure}[h]
\begin{center}
\begin{tikzpicture}[scale=0.45]
\tikzset{>=latex}
  \filldraw[color=gray!25] plot[smooth] coordinates {(0.15,4.7)(0.7,2.9)(4,0.4)(6,1.5)(9.5,0.4)(10,0.333)(10.01,0.331)(10.041,0.3) (10.041,-0.3)(10.01,-0.331)(10,-0.333)(9.5,-0.4)(6,-1.5)(4,-0.4)(0.7,-2.9)(0.15,-4.7)};
  \draw[thick] plot[smooth] coordinates {(0.15,4.7)(0.7,2.9)(4,0.4)(6,1.5)(9.5,0.4)(10,0.333)(10.01,0.331)(10.041,0.3)};
  \draw[thick] plot[smooth] coordinates {(10.041,-0.3)(10.01,-0.331)(10,-0.333)(9.5,-0.4)(6,-1.5)(4,-0.4)(0.7,-2.9)(0.15,-4.7)};
  \draw[thick,fill=lightgray] (0,0) ellipse (0.4 and 5);
  \draw[thick,fill=gray!25] (10.041,0) ellipse (0.1 and 0.333);
  \draw[thick] plot[smooth] coordinates {(0,2)(2,1.1)(4,-0.1)(6,-0.7)(9,0.25)(10,0.15)};
  \draw[thick,->] (-2,0)--(12,0) node[right,above]{\normalsize$t$};
  \node [black] at (0,2) {\textbullet};
  \draw[->,thick](3,3)node[right]{\normalsize$(0,e(0))$}--(0.07,2.07);
  \draw[->,thick](9,3)node[right]{\normalsize$1/\varphi(t)$}--(7,1.4);
\end{tikzpicture}
\end{center}
\vspace*{-2mm}
\caption{Error evolution in a funnel $\mathcal F_{\varphi}$ with boundary $1/\varphi(t)$.}
\vspace*{-4mm}
\label{Fig:funnel}
\end{figure}
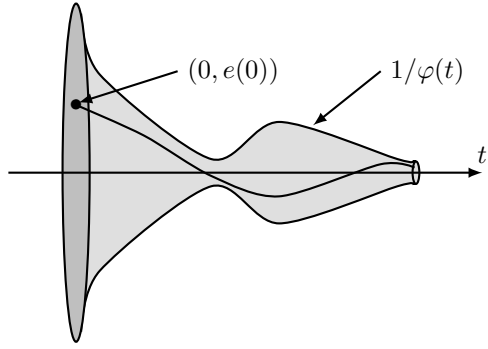

It is of utmost importance to notice that the function $\varphi\in\Phi$ is a design parameter in the control law (stated in Section~\ref{Sec:FunCon}), thus its choice is completely up to the designer. In particular, the designer must impose a priori, whether or not asymptotic tracking should be achieved. Typically, the specific application dictates the constraints on the tracking error and thus indicates suitable choices for~$\varphi$. We stress that the funnel boundary is not necessarily monotonically decreasing, while such a choice may be convenient in most situations. However, widening the funnel over some later time interval might be beneficial, for instance in the presence of strongly varying reference signals or periodic disturbances. A variety of different funnel boundaries are possible, see e.g.~\cite[Sec.~3.2]{Ilch13}.

\subsection{Organization of the present paper}

In Section~\ref{Sec:FPO} we introduce the mathematical framework around the Fokker-Planck operator associated to equation~\eqref{eq:FPE-OU}. We emphasize that we consider an unbounded spatial domain in~\eqref{eq:FPE-OU}, without any boundary conditions. Using weighted Lebesgue and Sobolev spaces, first an auxiliary operator is defined via a suitable sesquilinear form. This operator is then transformed to the desired Fokker-Planck operator. 
We stress that a spectral analysis of the Fokker-Planck operator is necessary in order to obtain a well-defined ``integration by parts''-formula, which in turn is required to show admissiblity of the bilinear control operator involved in~\eqref{eq:FPE-OU}. The definition of a mild solution is given in Section~\ref{Sec:SolProp} and it is shown that any solution satisfies~\eqref{eq:FPE-cond} and that its covariance matrix is independent of the input and exponentially converges to $c\Gamma^{-1}$. Furthermore, $L^2$-admissibility of the control operator is shown, which is the basis for the feasibility proof of the robust funnel controller in Section~\ref{Sec:FunCon}. A simple (non-robust) feedforward control approach is discussed in Section~\ref{Sec:NonRobCon}, which may be favourable when the system parameters are known and no disturbances are present. We emphasize that the closed-loop system corresponding to the application of the funnel controller, see equation~\eqref{eq:funnel-CL}, is a nonlinear and time-varying PDE, thus proving existence and uniqueness of solutions is a nontrivial task. We illustrate our results by some simulations in Section~\ref{Sec:Sim}.

\section{The Fokker-Planck operator}\label{Sec:FPO}

In this section we introduce an operator which can be associated with the PDE~\eqref{eq:FPE-OU} in the uncontrolled case, i.e., $u=0$. To this end, we invoke form methods for which we frequently refer to~\cite{ArenChil15} and~\cite{ArenElst12}. Consider the system~\eqref{eq:FPE-OU} with $\phi$ as defined in~\eqref{eq:phi}. To begin with, let
\[
    H:= L^2(\R^n;e^{-\phi})\quad \text{and}\quad V:=W^{1,2}(\R^n;e^{-\phi})
\]
and define the sesquilinear form
\begin{equation}\label{eq:form-a}
    a: V\times V\to \R,\ (v_1,v_2) \mapsto \sum_{i=1}^n \left\langle \frac{\partial v_1}{\partial x_i}, \frac{\partial v_2}{\partial x_i} \right\rangle_H = \langle \nabla v_1, \nabla v_2 \rangle_{H^n},
\end{equation}
to which we may associate an operator as follows.

\begin{proposition}\label{Prop:form-a-operator}
Consider the form~\eqref{eq:form-a}, then there exists exactly one operator $A : \cD(A) \subset V \to H$ with
\[
    \cD(A) = \setdef{ v\in V}{ \exists\, u\in H\ \forall\, z\in V:\ a(v,z) = \langle u,z\rangle_H}
\]
and
\[
    \forall\, v\in \cD(A)\ \forall\, z\in V:\ a(v,z) = \langle Av, z\rangle_H.
\]
Moreover, $A$ is self-adjoint, positive and has compact resolvent.
\end{proposition}
\begin{proof} We show that the operator~$A$ exists as stated. By the Cauchy-Schwarz inequality we have
\begin{align*}
    a(v,u) &\le \|\nabla v\|_{H^n} \|\nabla u\|_{H^n} \le \left( \|v\|_H^2 + \|\nabla v\|_{H^n}^2\right)^{1/2} \left( \|u\|_H^2 + \|\nabla u\|_{H^n}^2\right)^{1/2} = \|v\|_V \|u\|_V
\end{align*}
for all $v,u\in V$, and hence the form~$a$ is bounded. Since the injection $j:V\to H$ is clearly continuous with dense range, it follows from~\cite[Prop.~5.5]{ArenChil15} that~$A$ exists and is positive since~$a$ is positive.

We show~(i): As above, there exists an operator $B: \cD(B)\subset V\to H$ associated to the sesquilinear form
\[
    b: V\times V\to \R,\ (v_1,v_2) \mapsto a(v_1,v_2) + \langle v_1, v_2 \rangle_H
\]
which satisfies $\cD(B) = \cD(A)$ and $B = A + I$, cf.~\cite[Rem.~5.6]{ArenChil15}. The form~$b$ is obviously bounded and symmetric and satisfies $b(v,v) = \|v\|_V^2$, thus it is coercive. Further observe that by~\cite[Prop.~6.2]{John00} the injection $j:V\to H$ is additionally compact. Hence it follows from~\cite[Cor.~6.18]{ArenChil15} that the operator~$B$ is self-adjoint, positive and has compact resolvent. As a consequence, $A = B-I$ is also self-adjoint and has compact resolvent.
\end{proof}

Next we show that $-A$ is the generator of a $C_0$-semigroup with certain properties.

\begin{lemma}\label{Lem:A-generator}
The operator~$-A$ from Proposition~\ref{Prop:form-a-operator} generates an analytic contraction semigroup on~$H$.
\end{lemma}
\begin{proof}
By Proposition~\ref{Prop:form-a-operator} and the Lumer-Phillips theorem (see e.g.~\cite[Thms.~3.18~\&~6.1]{ArenChil15}) we find that $-A$ generates a contraction semigroup on~$H$. Further invoking~\cite[Thm.~4.3]{ArenElst12} we find that this semigroup is also analytic.
\end{proof}

In the following we explicitly derive the eigenvalues and eigenfunctions of~$A$. To this end, we first observe that the matrix $\Gamma$ is symmetric and positive definite, hence there exists an orthonormal basis $\{v_1,\ldots, v_n\}$ of $\R^n$ of eigenvectors of $\Gamma$ with $\Gamma v_k = \tilde \lambda_k v_k$ with $\tilde \lambda_k > 0$, $k=1,\ldots, n$. We define $\lambda_k := \tilde\lambda_k / c$ and $u_k := \sqrt{\lambda_k/2} v_k$ for $k=1,\ldots, n$. Since we have, for all $x\in\R^n$, that $\tfrac{1}{c} \Gamma x = \sum_{k=1}^n \lambda_k v_k^\top x v_k$, we record for later use that
\begin{equation}\label{eq:Gamma-uk}
    \frac{1}{c} \Gamma x = 2 \sum_{k=1}^n u_k^\top x u_k.
\end{equation}
Furthermore, recall the Hermite polynomials defined by
\[
    H_n(x) = (-1)^n e^{x^2} \left( \frac{{\rm d}^n}{{\rm d} x^n} e^{-x^2}\right),\quad x\in\R,\ n\in\N_0.
\]
It is well known that these polynomials have, for all $x\in\R$ and all $m,n\in\N_0$, the properties
\begin{enumerate}
  \item $H_{n+1}(x) = 2x H_n(x) - H_n'(x)$,
  \item $H_{n}'(x) = 2n H_{n-1}(x)$, where $H_{-1}(x) := 0$,
  \item $\int_{-\infty}^\infty e^{-x^2} H_n(x) H_m(x)\, {\rm d}x = \sqrt{\pi} 2^n n!\, \delta_{n,m}$, where $\delta_{n,m}$ denotes the Kronecker delta.
\end{enumerate}
Now let $\alpha \in (\N_0)^n$ be a multi-index. Then we define
\begin{equation}\label{eq:H-lambda-alpha}
  H_\alpha:\R^n\to\R,\ x\mapsto \prod_{k=1}^n H_{\alpha_k}(u_k^\top x),\quad
  \lambda_\alpha := \sum_{k=1}^n \alpha_k \lambda_k,
\end{equation}
which turn out to be the eigenfunctions and eigenvalues of $A$.

\begin{proposition}\label{Prop:eigenval-A}
Consider the operator~$A$ from Proposition~\ref{Prop:form-a-operator} and the eigenvectors~$u_k$ and eigenvalues~$\lambda_k$ of $c^{-1}\Gamma$. Then the spectrum of $A$ is given by
\[
    \sigma(A) = \setdef{\lambda_\alpha}{ \alpha \in (\N_0)^n}
\]
and the set $\setdef{ H_\alpha}{\alpha \in (\N_0)^n}$ constitutes a complete orthogonal system in~$H$ consisting of eigenfunctions of~$A$ with $A H_\alpha = \lambda_\alpha H_\alpha$ for all $\alpha \in (\N_0)^n$.
Furthermore, we have
\begin{enumerate}
  \item $\nabla \big( e^{-\phi(x)} H_\alpha(x)\big) = - e^{-\phi(x)} \sum_{k=1}^n \left( \prod_{j\neq k} H_{\alpha_j}(u_j^\top x)\right) H_{\alpha_k+1}(u_k^\top x) u_k$ for all $x\in\R^n$ and $\alpha \in (\N_0)^n$,
  \item $\lim_{r\to \infty} \int_{S_r} e^{-\phi(x)} H_\alpha(x) w(x) \cdot \vec{n}\, {\rm d}S = 0$ for all $w\in V^n$ and $\alpha \in (\N_0)^n$, where $S_r = \setdef{x\in\R^n}{\phi(x) = r}$ and $\vec{n}$ is the outward unit normal vector to its boundary.
\end{enumerate}
\end{proposition}
\begin{proof}
\emph{Step 1}: We first show~(ii), since it is needed for the other assertions. Fix $\alpha \in (\N_0)^n$, $j,k\in\{1,\ldots,n\}$ and $(x_1,\ldots,x_{k-1},x_{k+1},\ldots,x_n)\in\R^{n-1}$. Define the function
\[
    f_{j,k}:\R\to\R^n, x_k \mapsto e^{-3\phi(x)/4} H_\alpha(x) w_j(x).
\]
Since $H_\alpha$ is a polynomial we have that $\big(x_k\mapsto e^{-\phi(x)/4} H_\alpha(x)\big)\in L^\infty(\R)$. Furthermore, $w_j\in V$ yields that $e^{-\phi/2}w_j\in L^2(\R^n)$ and  $e^{-\phi/2} \tfrac{\partial w_j}{\partial x_k} \in L^2(\R^n)$. Hence, by Fubini's theorem we have that $\big(x_k\mapsto e^{-\phi(x)} w_j(x)^2\big), \big(x_k\mapsto e^{-\phi(x)} \left(\tfrac{\partial w_j}{\partial x_k}\right)(x)^2\big) \in L^1(\R)$. Therefore,
\begin{align*}
    &\left[x_k\mapsto f_{j,k}(x_k) = e^{-3\phi(x)/4} H_\alpha(x) w_j(x) = \left(e^{-\phi(x)/4} H_\alpha(x)\right)\left(e^{-\phi(x)/2}w_j(x)\right)\right] \in L^2(\R)\\
     \text{and}\  &\left[x_k\mapsto e^{-3\phi(x)/4} H_\alpha(x)  \left(\tfrac{\partial w_j}{\partial x_k}\right)(x) = \left(e^{-\phi(x)/4} H_\alpha(x)\right)\left(e^{-\phi(x)/2} \left(\tfrac{\partial w_j}{\partial x_k}\right)(x)\right)\right] \in L^2(\R).
\end{align*}
Moreover, we compute
\[
   \frac{\partial}{\partial x_k} \left(e^{-3\phi(x)/4} H_\alpha(x)\right) = -\frac{3}{4c} e_k^\top \Gamma x  e^{-3\phi(x)/4} H_\alpha(x) + e^{-3\phi(x)/4} \frac{\partial H_\alpha}{\partial x_k}(x) = e^{-3\phi(x)/4} p_{\alpha,k}(x),
\]
where $p_{\alpha,k}$ is some polynomial, whose degree depends on $\alpha$ and $k$. In any case, we have that $\big(x_k\mapsto e^{-\phi(x)/4} p_{\alpha,k}(x)\big) \in L^\infty(\R)$, thus
\[
    \frac{\partial f_{j,k}}{\partial x_k}(x_k) =   \left(e^{-\phi(x)/4} p_{\alpha,k}(x)\right)\left(e^{-\phi(x)/2}w_j(x)\right)  + e^{-3\phi(x)/4} H_\alpha(x) \left(\tfrac{\partial w_j}{\partial x_k}\right)(x) \in L^2(\R).
\]
Since $f_{j,k}, f_{j,k}'\in L^2(\R)$, it follows from Barb\u{a}lat's Lemma (see e.g.~\cite[Thm.~5]{FarkWegn16}) that $\lim_{x_k\to \pm\infty} f_{j,k}(x_k) = 0$. Since this is true for all $j,k\in\{1,\ldots,n\}$ it is easily seen that $e^{-3\phi/4} H_\alpha w \in L^\infty(\R^n)^n$ follows. Now we may observe that
\begin{align*}
    \int_{S_r} e^{-\phi(x)} H_\alpha(x) w(x) \cdot \vec{n}\, {\rm d}S &= \int_{S_r} e^{-r/4} e^{-3\phi(x)/4}H_\alpha(x) w(x) \cdot \vec{n}\, {\rm d}S\\
      &\le \sum_{k=1}^n M_k \int_{S_r} e^{-r/4} e_k \cdot \vec{n}\, {\rm d}S \le K e^{-r/4} r^{n-1},
\end{align*}
for some constants $M_1,\ldots,M_n,K>0$. This implies assertion~(ii).

\emph{Step 2}: We show that $H_\alpha\in\cD(A)$ and $A H_\alpha = \lambda_\alpha H_\alpha$ for all $\alpha \in (\N_0)^n$. First note that $H_\alpha \in V$ since, using the properties of the Hermite polynomials,
\begin{equation}\label{eq:nabla-Halpha}
  \nabla H_\alpha(x) = 2 \sum_{k=1}^n \left( \prod_{j\neq k} H_{\alpha_j}(u_j^\top x)\right) \alpha_k H_{\alpha_k-1}(u_k^\top x) u_k,\quad x\in\R^n.
\end{equation}
By definition of~$A$ the two assertions hold if, and only if, $a(H_\alpha,z) = \lambda_\alpha \langle H_\alpha, z\rangle_H$ for all $z\in V$. For $\alpha=(0,\ldots,0)$ this is clear since $\lambda_\alpha=0$ in this case, and $H_\alpha(x) = 1$, thus $a(H_\alpha,z) = 0$ for all $z\in V$. Now, fix $\alpha \in (\N_0)^n$ and $z\in V$, {and define the multi-index $\alpha^{-i}$ by
\[
    \alpha^{-i}_j := \begin{cases} \alpha_j, & j\neq i,\\ \alpha_i-1, & j=i,\end{cases}\quad j=1,\ldots,n.
\]
Then we have $\nabla H_\alpha(x) \stackrel{\eqref{eq:nabla-Halpha}}{=} 2 \sum_{k=1}^n \alpha_k H_{\alpha^{-k}}(x) u_k$ for all $x\in\R^n$ and
\begin{align*}
    & a(H_\alpha,z) = \int_{\R^n} e^{-\phi(x)} \big(\nabla H_\alpha(x)\big)^\top \big(\nabla  z(x)\big) {\rm d}x = \lim_{r\to\infty} \int_{S_r} e^{-\phi(x)} z(x) \nabla H_\alpha(x) \cdot \vec{n}\, {\rm d}S \\ &- \int_{\R^n} z(x) \div\left(e^{-\phi(x)} \nabla H_\alpha(x)\right)  {\rm d}x
      = 2 \sum_{k=1}^n \alpha_k \lim_{r\to\infty} \int_{S_r} e^{-\phi(x)}   H_{\alpha^{-k}}(x)  \big( z(x)u_k\big) \cdot \vec{n}\, {\rm d}S \\ &- \int_{\R^n} z(x) \div\left(e^{-\phi(x)} \nabla H_\alpha(x)\right)  {\rm d}x \stackrel{\rm Step~1}{=}
    - \int_{\R^n} z(x) \div\left(e^{-\phi(x)} \nabla H_\alpha(x)\right)  {\rm d}x
\end{align*}
}
and we compute
\begin{align*}
\div\left(e^{-\phi(x)} \nabla H_\alpha(x)\right)&= e^{-\phi(x)} \left( - \frac{1}{c} x^\top \Gamma \nabla H_\alpha(x) + \div\left(   \nabla H_\alpha(x)\right)\right) \\
&\stackrel{\eqref{eq:Gamma-uk}}{=} e^{-\phi(x)} \left(\div\left(\nabla H_\alpha(x)\right) -  2 \sum_{k=1}^n (u_k^\top x) u_k^\top \nabla H_\alpha(x)\right)
\end{align*}
and
\begin{align*}
  \div\left(\nabla H_\alpha(x)\right) &\stackrel{\eqref{eq:nabla-Halpha}}{=} 2 \sum_{\ell=1}^n \sum_{k=1}^n \alpha_k \left( \left( \prod_{j\neq k} H_{\alpha_j}(u_j^\top x)\right) H_{\alpha_k-1}'(u_k^\top x) u_{k,\ell}^2 \right.\\ &\quad \left.+ H_{\alpha_k-1}(u_k^\top x) u_{k,\ell} \sum_{m\neq k} \left( \prod_{j\not\in\{k,m\}} H_{\alpha_j}(u_j^\top x)\right) H_{\alpha_m}'(u_m^\top x) u_{m,\ell}\right)\\
  &= 2 \sum_{k=1}^n \alpha_k \left( \left( \prod_{j\neq k} H_{\alpha_j}(u_j^\top x)\right) H_{\alpha_k-1}'(u_k^\top x) \|u_k\|_{\R^n}^2\right.\\
   &\quad \left.+ H_{\alpha_k-1}(u_k^\top x) \sum_{m\neq k} \left( \prod_{j\not\in\{k,m\}} H_{\alpha_j}(u_j^\top x)\right) H_{\alpha_m}'(u_m^\top x) u_k^\top u_m\right) \\
  &= 2 \sum_{k=1}^n \alpha_k  \left( \prod_{j\neq k} H_{\alpha_j}(u_j^\top x)\right) H_{\alpha_k-1}'(u_k^\top x) \|u_k\|_{\R^n}^2,
\end{align*}
since $u_k^\top u_m = 0$ for $k\neq m$. Therefore, we obtain, using the properties of the Hermite polynomials {and that by definition of $u_k$ we have $\|u_k\|_{\R^n}^2 = \tfrac{\lambda_k}{2}$},
\begin{align*}
&\div\left(e^{-\phi(x)} \nabla H_\alpha(x)\right)\\
&= 2 e^{-\phi(x)} \sum_{k=1}^n \left( \prod_{j\neq k} H_{\alpha_j}(u_j^\top x)\right) \alpha_k \|u_k\|_{\R^n}^2 \left(  H_{\alpha_k-1}'(u_k^\top x)   - 2 (u_k^\top x) H_{\alpha_k-1}(u_k^\top x)\right) \\
&=  -2 e^{-\phi(x)} \sum_{k=1}^n \left( \prod_{j\neq k} H_{\alpha_j}(u_j^\top x)\right) \alpha_k \|u_k\|_{\R^n}^2 H_{\alpha_k}(u_k^\top x)\\
& = - e^{-\phi(x)} H_\alpha(x) \sum_{k=1}^n \alpha_k \lambda_k = -\lambda_\alpha  e^{-\phi(x)} H_\alpha(x),
\end{align*}
and hence, finally,
\[
    a(H_\alpha,z) = - \int_{\R^n} z(x) \div\left(e^{-\phi(x)} \nabla H_\alpha(x)\right)  {\rm d}x = \lambda_\alpha \langle H_\alpha, z\rangle_H.
\]

\emph{Step 3}: As shown in~\cite{Xu17} the products of Hermite polynomials constitute a complete orthogonal system in $L^2(\R^n;w)$ for $w(x) = e^{-\|x\|_{\R^n}^2}$. Since $\phi(x) = \sum_{k=1}^n (u_k^\top x)^2$ by~\eqref{eq:Gamma-uk} if follows (after defining new coordinates $y_k = u_k^\top x$) that $\setdef{H_\alpha}{\alpha \in (\N_0)^n}$ constitutes a complete orthogonal system in~$H$. This also implies that $\sigma(A) = \setdef{\lambda_\alpha}{ \alpha \in (\N_0)^n}$.

\emph{Step 4}: We show~(i). Observe that
\begin{align*}
    &\nabla \big(e^{-\phi(x)} H_\alpha(x)\big) = e^{-\phi(x)} \big( -\nabla\phi(x)  H_\alpha(x) + \nabla H_\alpha(x)\big) = e^{-\phi(x)} \big( \nabla H_\alpha(x) - H_\alpha(x) c^{-1} \Gamma x\big) \\
    &\stackrel{\eqref{eq:Gamma-uk},\eqref{eq:nabla-Halpha}}{=}\!  e^{-\phi(x)} \left(  \sum_{k=1}^n \left( \prod_{j\neq k} H_{\alpha_j}(u_j^\top x)\right) 2 \alpha_k H_{\alpha_k-1}(u_k^\top x) u_k - \prod_{j=1}^n H_{\alpha_j}(u_j^\top x) \left(\sum_{k=1}^n 2 u_k^\top x u_k\right)\!\! \right)\\
    &=  e^{-\phi(x)} \left(  \sum_{k=1}^n \left( \prod_{j\neq k} H_{\alpha_j}(u_j^\top x)\right) \big( H_{\alpha_k}'(u_k^\top x) - 2 u_k^\top x H_{\alpha_k}(u_k^\top x)  \big) u_k\right) \\
    &=  -e^{-\phi(x)} \sum_{k=1}^n \left( \prod_{j\neq k} H_{\alpha_j}(u_j^\top x)\right) H_{\alpha_k+1}(u_k^\top x) u_k,
\end{align*}
where we have used the properties of the Hermite polynomials.
\end{proof}

Now we turn to transform the operator~$A$ so that it becomes a suitable Fokker-Planck operator. To this end, define the spaces
\begin{align*}
  \fH &:= \setdef{e^{-\phi} f}{f\in H} = L^2(\R^n;e^{\phi}),\quad
  \fV := \setdef{e^{-\phi}f}{f\in V}
\end{align*}
and the bijection $h:H\to \fH,\ f\mapsto e^{-\phi} f$, together with the inner products
\begin{align*}
  \langle z_1, z_2 \rangle_\fH &:= \langle h^{-1}(z_1),h^{-1}(z_2) \rangle_H = \langle e^\phi z_1, e^\phi z_2 \rangle_H, && z_1, z_2\in\fH,\\
  \langle z_1, z_2 \rangle_\fV &:= \langle h^{-1}(z_1),h^{-1}(z_2) \rangle_V = \langle e^\phi z_1, e^\phi z_2 \rangle_H + \left\langle \nabla  (e^\phi z_1), \nabla (e^\phi z_2) \right\rangle_{H^n}, && z_1, z_2\in\fV.
\end{align*}
Further define the sesquilinear form
\begin{equation}\label{eq:form-fa}
    \fa:\fV\times\fV\to \R,\ (z_1,z_2)\mapsto a\left(h^{-1}(z_1),h^{-1}(z_2)\right) = \langle \nabla (e^\phi z_1), \nabla (e^\phi z_2) \rangle_{H^n},
\end{equation}
as well as $\cD(\fA) := h(\cD(A))$ and the operator
\[
    \fA := h \circ A \circ h^{-1} : \cD(\fA) \subset \fV \to \fH.
\]
Then we have that, for $v\in\cD(\fA)$ and $y\in\fH$,
\begin{align*}
  y = \fA v\quad &\iff\quad h^{-1}(y) = A h^{-1}(v)\quad \iff \quad \forall\, z\in V:\ a\big( h^{-1}(v), z\big) = \langle h^{-1}(y), z\rangle_H\\
  &\stackrel{w=h(z)}{\iff} \quad \forall\, w\in \fV:\ \fa\big(v,w\big) = \langle y, w\rangle_\fH.
\end{align*}
Furthermore, it is easy to see that $\fA$ is symmetric and that $\cD(\fA^*) = h(\cD(A^*)) = h(\cD(A))= \cD(\fA)$, thus $\fA$ is self-adjoint. From Proposition~\ref{Prop:eigenval-A} we immediately obtain the following result on the eigenvalues and eigenfunctions of~$\fA$.

\begin{proposition}\label{Prop:eigenval-frakA}
The operator~$\fA$ is self-adjoint and positive and satisfies
\begin{enumerate}
  \item $\sigma(\fA) = \sigma(A)$ and $z$ is an eigenfunction of~$\fA$ if, and only if, $e^{\phi} z$ is an eigenfunction of~$A$,
  \item for $z_\alpha := e^{-\phi} H_\alpha$ the set $\setdef{ z_\alpha}{\alpha \in (\N_0)^n}$ constitutes a complete orthogonal system of eigenfunctions in~$\fH$ with $\fA z_\alpha = \lambda_\alpha z_\alpha$,
  \item $\lim_{r\to \infty} \int_{S_r} e^{\phi(x)} z_\alpha(x) w(x) \cdot \vec{n}\, {\rm d}S = 0$ for all $w\in \fV^n$ and $\alpha \in (\N_0)^n$.
\end{enumerate}
\end{proposition}

Attention now turns to the operator~$-c\fA$, which will serve as the Fokker-Planck operator. In view of the right-hand side in~\eqref{eq:FPE-OU}, this is justified by the following property.

\begin{lemma}\label{Lem:fA-FPO}
Let $z\in \fV$ be such that $\nabla (e^{\phi} z)\in V^n$. Then we have that
\[
   \fA z = -\div \left(e^{-\phi} \nabla \left(e^{\phi} z\right)\right) = -\div \left( \nabla z + z \nabla \phi\right).
\]
\end{lemma}
\begin{proof} Let $(z_\alpha)_{\alpha\in(\N_0)^n}$ be the eigenfunctions of~$\fA$ from Proposition~\ref{Prop:eigenval-frakA}. We calculate that for any $\alpha\in(\N_0)^n$
\begin{align*}
  &\langle  \fA z, z_\alpha\rangle_\fH = \fa\big(z,z_\alpha\big) = \int_{\R^n} e^{-\phi(x)} \nabla \left(e^{\phi(x)} z(x)\right)^\top \nabla \left(e^{\phi(x)} z_\alpha(x)\right) {\rm d}x \notag \\
  &=\lim_{r\to \infty} \int_{S_r} e^{-\phi(x)} H_\alpha(x) \nabla \left(e^{\phi(x)} z(x)\right) \cdot \vec{n}\, {\rm d}S -\! \int_{\R^n} e^{\phi(x)} \div \left(e^{-\phi(x)} \nabla \left(e^{\phi(x)} z(x)\right)\right) z_\alpha(x) {\rm d}x \notag \\
  &= \left\langle -\div \left(e^{-\phi(x)} \nabla \left(e^{\phi(x)} z(x)\right)\right), z_\alpha\right\rangle_\fH, \label{eq:repr-fA}
\end{align*}
where the last equality follows from the assumption  $\nabla (e^{\phi} z)\in V^n$ and
Proposition~\ref{Prop:eigenval-A}\,(ii). Since the above equality is true for all $\alpha\in(\N_0)^n$, we have proved the first equality in the statement. The second is a straightforward calculation.
\end{proof}

Recall that $c \nabla \phi(x) = \Gamma x$ for all $x\in\R^n$. Therefore, with the operator
\begin{equation}\label{eq:def-opB}
    \fB : \fH \times\R^n \to \fH_{-1},\ (v,u) \mapsto -\div \big(v \cdot g(u)\big),
\end{equation}
for which it is clear that $\fB(\cdot, u)\in \cB(\fH;\fH_{-1})$ for all $u\in\R^n$, the Fokker-Planck equation~\eqref{eq:FPE-OU} can be rewritten as
\begin{equation}\label{eq:FPE-OU-fA}
\begin{aligned}
 \dot p(t,x) &= -c\fA p(t,x) + \fB\big(p(t,\cdot),u(t)\big)(x), &&\text{in}\ (0,\infty)\times \R^n,\\
  p(0,x) &= p_0(x), &&\text{in}\ \R^n,
\end{aligned}
\end{equation}
with state space~$\fH$. Note that the space~$\fH_{-1}$ is defined with respect to the Fokker-Planck operator  $-c\fA$. System~\eqref{eq:FPE-OU-fA} fits into the framework of bilinear control systems as considered for the Fokker-Planck equation e.g.\ in~\cite{BreiKuni18,HosfJaco20}. Although it has been considered only on a bounded spatial domain in the aforementioned works, the results for general bilinear systems from~\cite{HosfJaco20} may still be used to infer the existence of a unique mild solution to the open-loop problem. This will be one ingredient in our analysis of the closed-loop system under funnel control, see Section~\ref{Sec:FunCon}.


\section{Mild solutions and their properties}\label{Sec:SolProp}

In this section we introduce the notion of mild solutions of the Fokker-Planck equation~\eqref{eq:FPE-OU-fA}, where we closely follow the framework for bilinear systems introduced in~\cite{HosfJaco20}. We show admissibility of the involved control operators and derive a set of properties that each solution exhibits, including a covariance matrix independent of the control input and properties~\eqref{eq:FPE-cond}.

First, we introduce
\begin{equation}\label{eq:B1B2F}
  \fB_1 : \fH^n \to \fH_{-1},\ v\mapsto - \div v,\quad \fB_2 = I_{\fH},\quad
  F: \fH \times \R^n \to \fH^n,\ (v,u)\mapsto v \cdot g(u) ,
\end{equation}
where we recall that $g\in C^1(\R^n;\R^n)$ is such that~\eqref{eq:HGP} is satisfied and $\fH_{-1}$ is defined w.r.t.\ $-c\fA$. Further let $d\in L^\infty(0,\infty;\fH)$ be a disturbance that has ``zero mass'' in the sense
\begin{equation}\label{eq:cond-dist}
    \int_{\R^n} d(t,x) {\rm d}x = 0\quad \text{for almost all $t\ge0$.}
\end{equation}
We may observe that the above condition is equivalent to $\langle d(t), e^{-\phi}\rangle_\fH = 0$, i.e., the disturbance is restricted to the orthogonal complement of the eigenfunction corresponding to the zero eigenvalue of~$\fA$. Thus, it influences only the exponentially stable part of the Fokker-Planck operator. We introduce~$d$ as an additive {and unknown} disturbance in the Fokker-Planck equation~\eqref{eq:FPE-OU-fA}, which may be restated as, omitting the argument~$x$,
\begin{equation}\label{eq:FPE-OU-dist}
 \dot p(t) = -c\fA p(t) + \fB_1 F\big(p(t),u(t)\big) + \fB_2 d(t),\quad p(0) = p_0.
\end{equation}
{Since a model does typically not exactly describe a real-world process, the disturbance can, for instance, be understood as the uncertainty which distinguishes the ideal model (i.e., with $d=0$) from the process at hand.} Note that, in the presence of disturbances, it cannot be expected that the solution~$p(t)$ is a probability density function for~$t\ge 0$ in general, i.e., conditions~\eqref{eq:FPE-cond} will typically not hold.

Before defining the mild solution we recall from Lemma~\ref{Lem:A-generator} that $-A$ generates an analytic contraction semigroup on~$H$. Therefore, also the Fokker-Planck operator $-c\fA$ generates an analytic contraction semigroup on~$\fH$ denoted by $(T(t))_{t\ge 0}$ in the following.

\begin{definition}\label{Def:solution}
Consider the system~\eqref{eq:FPE-OU-dist} with $c>0$, $\Gamma = \Gamma^\top>0$, $g\in C^1(\R^n;\R^n)$ with~\eqref{eq:HGP} and $\phi$ as defined in~\eqref{eq:phi}. Recall the spaces~$\fH$ and~$\fV$ from Section~\ref{Sec:FPO} and let~$p_0\in \fH$, $t_1>0$ and $u\in C(\R_{\ge 0};\R^n)$, $d\in L^\infty(0,\infty;\fH)$. A function~$p\in C([0,t_1]; \fH)$ is called \emph{mild solution} of~\eqref{eq:FPE-OU-dist} on $[0,t_1]$, if
\begin{equation}\label{eq:mild_solution}
    p(t) = T(t)p_0 + \int_0^t T_{-1}(t-s) \big(\fB_1 F\big(p(s),u(s)\big) + \fB_2 d(s)\big) {\rm d}s
\end{equation}
for all $t\in [0,t_1]$. A function~$p$ is called mild solution of~\eqref{eq:FPE-OU-dist} on $\R_{\ge 0}$, if $p|_{[0,t_1]}$ is a solution of~\eqref{eq:FPE-OU-dist} on $[0,t_1]$ for all $t_1>0$.
\end{definition}

We note that, while the function in~\eqref{eq:mild_solution} clearly satisfies $p(t) \in\fH_{-1}$ for $t\ge 0$, $p\in C([0,t_1]; \fH)$ is an additional condition. To achieve this property the concept of admissibility is used, see the Nomenclature. In the following we show that the control operator of the bilinear system~\eqref{eq:FPE-OU-dist} is admissible, where we follow the ideas given in~\cite[Sec.~3]{HosfJaco20}, tailored to the present framework. To this end, with respect to the Fokker-Planck operator $-c\fA$, which is self-adjoint and negative by Proposition~\ref{Prop:eigenval-frakA}, we introduce the space $\fH_{\frac12}$ as the completion of $\cD(\fA)$ with respect to the norm
\[
    \|v\|_{\fH_{\frac12}}^2  = \langle (I+c\fA)v, v\rangle_\fH,\quad v\in\cD(\fA).
\]
Furthermore, the space $\fH_{-\frac12}$ is defined as the completion of~$\fH$ with respect to the norm
\[
    \|v\|_{\fH_{-\frac12}}  = \sup_{\|w\|_{\fH_{\frac12}}\le 1} |\langle v,w\rangle_\fH|,\quad v\in \fH.
\]
It is easy to see that, for all $v\in\fH_{\frac12}$,
\begin{equation}\label{eq:norm-H12}
  \|v\|_{\fH_{\frac12}}^2 = \|v\|_\fH^2 + c \fa(v,v) = \|v\|_\fH^2 + c \|e^{-\phi} \nabla (e^{\phi} v)\|_{\fH^n}^2,
\end{equation}
and since $\|v\|_\fV^2 = \|v\|_\fH^2 + \fa(v,v)$, we have $\fH_{\frac12} = \fV$ with different, but equivalent, norms.

\begin{lemma}\label{Lem:admiss}
We have that $\fB_1 \in \cB(\fH^n; \fH_{-\frac12})$ and~$\fB_1$ is $L^2$-admissible for $(T(t))_{t\ge 0}$.
\end{lemma}
\begin{proof} Combining~\cite[Thms.~4.4.3~\&~5.1.3]{TucsWeis09} it follows that $\fB_1$ is $L^2$-admissible for $(T(t))_{t\ge 0}$, if $\fB_1\in \cB(\fH^n,\fH_{-\frac12})$. To show the latter, let $w\in\fV^n$, $\alpha\in (\N_0)^n$ and, invoking $e^{-\phi} \nabla H_\alpha \in \fH^n$, $\tilde H_\alpha := H_\alpha / \|e^{-\phi} H_\alpha\|_{\fH_{\frac12}}$, then
\begin{align*}
    |\langle \fB_1 w, e^{-\phi} \tilde H_\alpha\rangle_\fH| &= \left| \int_{\R^n} \tilde H_{\alpha}(x) \div w(x) {\rm d}x\right| = \left| \int_{\R^n} w(x)^\top \nabla \tilde H_{\alpha}(x) {\rm d}x\right| \\
    &\le \|w\|_{\fH^n} \|e^{-\phi} \nabla \tilde H_\alpha\|_{\fH^n} \stackrel{\eqref{eq:norm-H12}}{\le} \frac{1}{\sqrt{c}} \|w\|_{\fH^n} \|e^{-\phi} \tilde H_\alpha\|_{\fH_{\frac12}} = \frac{1}{\sqrt{c}} \|w\|_{\fH^n},
\end{align*}
where we have used Proposition~\ref{Prop:eigenval-A}\,(ii). Therefore, we find that
\begin{align*}
  \|\fB_1 w\|_{\fH_{-\frac12}}  &= \sup_{\|v\|_{\fH_{\frac12}}\le 1} |\langle \fB_1 w,v\rangle_\fH| = \sup_{\alpha\in (\N_0)^n} |\langle \fB_1 w, e^{-\phi} \tilde H_\alpha\rangle_\fH| \le \frac{1}{\sqrt{c}} \|w\|_{\fH^n},
\end{align*}
and since $\fV^n$ is dense in $\fH^n$ it follows that $\fB_1\in \cB(\fH^n,\fH_{-\frac12})$.
\end{proof}

Next we show that any mild solution of~\eqref{eq:FPE-OU-dist} satisfies the equation in the weak sense and exhibits a certain smoothness. First recall that $\fH_{-{\frac12}}$ is the dual of $\fH_{\frac12}$ with respect to the pivot space $\fH$, thus, invoking that $\fA$ is self-adjoint and using an appropriate identification via the Riesz representation theorem, we have
\[
    \langle w, v\rangle_{\fH_{-{\frac12}}\times \fH_{\frac12}} = \langle w,v\rangle_\fH,\quad w\in\fH_{-{\frac12}},\ v\in\fH_{\frac12},
\]
i.e., the duality pairing is compatible with the inner product in~$\fH$, cf.\ also~\cite[Sec.~3.6]{Staf05} and~\cite[Sec.~2.9]{TucsWeis09}.

\begin{lemma}\label{Lem:sln-weak}
Use the assumptions from Definition~\ref{Def:solution} and let~$p$ be a mild solution of~\eqref{eq:FPE-OU-dist} on $[0,t_1]$. Then $p\in L^q(0,t_1;\fV) \cap W^{1,q}(0,t_1;\fH_{-{\frac12}})$ for all $1\le q<2$ and for all $v\in\fV$ and almost all $t\in[0,t_1]$ we have
\begin{equation}\label{eq:weak-repr}
    \langle \dot p(t), v\rangle_\fH = -c\langle p(t), \fA v\rangle_\fH + \langle \fB_1 F\big(p(t),u(t)\big) + \fB_2 d(t), v\rangle_\fH.
\end{equation}
If additionally $p_0 \in \fV$, then $p\in L^q(0,t_1;\fV) \cap W^{1,q}(0,t_1;\fH_{-{\frac12}})$ for all $1\le q<\infty$.
\end{lemma}
\begin{proof} Fix $1<q<\infty$. First we conclude from~\cite[Thm.~3.10.11]{Staf05} that the analytic semigroup~$(T(t))_{t\ge 0}$ generated by $-c\fA$ on~$\fH$ extends to an analytic semigroup~$(T_{-\frac12}(t))_{t\ge 0}$ on~$\fH_{-\frac12}$ with generator $-c\fA_{-\frac12}$. Since~$\fH_{-\frac12}$ is again a Hilbert space, the analytic semigroup~$(T_{-\frac12}(t))_{t\ge 0}$ has the maximal regularity property as shown in~\cite{DeSi64}, cf.\ also~\cite{ArenBu03}. This means that, in particular,
\begin{equation}\label{eq:MR}
    \forall\, f\in L^q(0,t_1;\fH_{-\frac12}):\ x\in W^{1,q}(0,t_1;\fH_{-\frac12})\ \wedge\ \fA_{-\frac12} x \in L^q(0,t_1;\fH_{-\frac12}),
\end{equation}
where~$x$ denotes the mild solution of the Cauchy problem
\begin{equation}\label{eq:CP}
  \dot x(t) = -c\fA_{-\frac12} x(t) + f(t),\quad x(0)=0
\end{equation}
in~$\fH_{-\frac12}$, that is $x(t) = \int_0^t T_{-\frac12}(t-s) f(s) {\rm d}s$ for $t\in [0,t_1]$. {Recall~\eqref{eq:B1B2F} and} define
\[
    f(t):= \fB_1 F\big(p(t),u(t)\big) + \fB_2 d(t),\quad t\in[0,t_1],
\]
then it follows from $p\in C([0,t_1];\fH)$, $u\in C([0,t_1];\R^n)$ and $\fB_1 \in \cB(\fH^n; \fH_{-\frac12})$ by Lemma~\ref{Lem:admiss} that $f\in L^q(0,t_1;\fH_{-\frac12})$, where we have used that $\|d(t)\|_{\fH_{-\frac12}} \le \|d(t)\|_\fH$ by~\eqref{eq:norm-H12}. Therefore, property~\eqref{eq:MR} implies that
\[
    \tilde p(\cdot):= \int_0^{\cdot} T_{-\frac12}(\cdot-s) f(s) {\rm d}s\quad\text{satisfies}\quad \tilde p\in W^{1,q}(0,t_1;\fH_{-\frac12})\ \wedge\ \fA_{-\frac12} \tilde p \in L^q(0,t_1;\fH_{-\frac12}).
\]
We calculate, for $t\in[0,t_1]$,
\begin{align*}
  &\|\tilde p(t)\|^2_{\fH_{\frac12}} = \langle (I+c\fA)\tilde p(t),\tilde p(t)\rangle_\fH = \|\tilde p(t)\|_\fH^2 + c \left\langle \fA_{-\frac12} \tilde p(t),\frac{\tilde p(t)}{\|\tilde p(t)\|_{\fH_{\frac12}}}\right\rangle_\fH \|\tilde p(t)\|_{\fH_{\frac12}}\\
  &\le \|\tilde p(t)\|_\fH^2 + c \| \fA_{-\frac12} \tilde p(t)\|_{\fH_{-\frac12}} \|\tilde p(t)\|_{\fH_{\frac12}} \le \|\tilde p(t)\|_\fH^2 +\frac{c^2}{2} \| \fA_{-\frac12} \tilde p(t)\|_{\fH_{-\frac12}}^2 + \frac12 \|\tilde p(t)\|_{\fH_{\frac12}}^2, \\
\end{align*}
which gives
\[
    \|\tilde p(t)\|_{\fH_{\frac12}} \le \left(2\|\tilde p(t)\|_\fH^2 + c^2 \| \fA_{-\frac12} \tilde p(t)\|_{\fH_{-\frac12}}^2\right)^{1/2}\le \sqrt{2} \|\tilde p(t)\|_\fH + c \| \fA_{-\frac12} \tilde p(t)\|_{\fH_{-\frac12}}.
\]
Since $T_{-1}(t-s) f(s) = T_{-\frac12}(t-s) f(s)$ we have that $p(t) = T(t) p_0 + \tilde p(t)$ for all $t\in [0,t_1]$ and, as~$p$ is a mild solution, $\tilde p\in C([0,t_1];\fH)$, which also gives $\tilde p\in L^q(0,t_1;\fH)$. Therefore, we have
\[
    \|\tilde p\|^q_{L^q(0,t_1;\fH_{\frac12})} \le 2^{\frac{q-1}{q}} \left( 2^{\frac{q}{2}} \|\tilde p\|^q_{L^q(0,t_1;\fH)} + c^q \| \fA_{-\frac12} \tilde p\|_{L^q(0,t_1;\fH_{-\frac12})}^q\right),
\]
by which $\tilde p \in L^q(0,t_1;\fH_{\frac12})$. Attention now turns to the term $T(t) p_0$. As~$(T(t))_{t\ge 0}$ is analytic it follows from~\cite[Thm.~3.10.6]{Staf05} that
\[
   \exists\,M>0\ \forall\, t>0:\ \|c\fA T(t) p_0\|_\fH \le \frac{M}{t} \|p_0\|_\fH.
\]
Therefore, we find that, using the inner product $\langle v,w\rangle_{\fH_{\frac12}} = \langle v,w\rangle_\fH + c \fa(v,w)$ in $\fH_{\frac12}$,
\begin{align*}
  &\|\ddt T(t)p_0\|_{\fH_{-\frac12}} = \sup_{\|w\|_{\fH_{\frac12}}\le 1} \left|\left\langle c\fA T(t) p_0,w\right\rangle_\fH\right|= \sup_{\|w\|_{\fH_{\frac12}}\le 1} \left|\left\langle T(t) p_0,w\right\rangle_{\fH_{\frac12}} - \left\langle T(t) p_0,w\right\rangle_\fH\right| \\
  &\stackrel{\eqref{eq:norm-H12}}{\le} \sup_{\|w\|_{\fH_{\frac12}}\le 1} \left( \|T(t) p_0\|_{\fH_{\frac12}} \|w\|_{\fH_{\frac12}} + \|T(t) p_0\|_\fH \|w\|_{\fH_{\frac12}}\right) \\
  &\stackrel{\eqref{eq:norm-H12}}{\le} \sqrt{\|T(t) p_0\|_\fH^2 + \|c\fA T(t)p_0\|_\fH \|T(t)p_0\|_\fH} + \|T(t) p_0\|_\fH \le \left(1 + \sqrt{1 + \tfrac{M}{t}} \right) \|p_0\|_\fH
\end{align*}
for $t>0$, where we have used that~$(T(t))_{t\ge 0}$ is a contraction semigroup. Therefore $\ddt T(\cdot)p_0 \in L^q(0,t_1;\fH_{-\frac12})$ for all $1\le q < 2$. Together with $\|T(t)p_0\|_{\fH_{-\frac12}} \le \|T(t)p_0\|_{\fH} \le \|p_0\|_\fH$ for $t\ge 0$ this implies that $T(\cdot)p_0 \in W^{1,q}(0,t_1;\fH_{-\frac12})$ for $1\le q < 2$. Moreover, in the above inequality we have used that $\|T(t) p_0\|_{\fH_{\frac12}}\le \sqrt{1 + \tfrac{M}{t}} \|p_0\|_\fH$ for $t>0$, by which $T(\cdot)p_0 \in L^q(0,t_1;\fH_{\frac12})$ for all $1\le q <2$. Together with the findings on~$\tilde p$ we thus obtain $p\in L^q(0,t_1;\fH_{\frac12}) \cap W^{1,q}(0,t_1;\fH_{-{\frac12}})$ for all $1\le q<2$.

If $p_0\in \fV = \fH_{\frac12}$, then it follows from~\cite[Prop.~4.2.5]{TucsWeis09} (with $X=\fH_{\frac12}$ and $B=0$) that $T(\cdot)p_0 \in C([0,\infty);\fH_{\frac12})$, by which $T(\cdot)p_0 \in L^\infty(0,t_1;\fH_{\frac12})$. Since $\|\ddt T(t)p_0\|_{\fH_{-\frac12}} \le \|T(t) p_0\|_{\fH_{\frac12}} + \|p_0\|_\fH$ it further follows $T(\cdot)p_0 \in W^{1,\infty}(0,t_1;\fH_{-\frac12})$ and together with the findings on~$\tilde p$ this gives $p\in L^q(0,t_1;\fH_{\frac12}) \cap W^{1,q}(0,t_1;\fH_{-{\frac12}})$ for all $1\le q<\infty$.

Finally, since $p\in W^{1,1}(0,t_1;\fH_{-{\frac12}})$ we find that it satisfies~\eqref{eq:FPE-OU-dist} pointwise almost everywhere in~$\fH_{-{\frac12}}$, which gives~\eqref{eq:weak-repr}.
\end{proof}

\begin{remark}
  Note that it is possible to extend the regularity results from Lemma~\ref{Lem:sln-weak} to obtain statements in terms of the spaces of H\"{o}lder continuous functions using the theory from~\cite{Luna95}. Then, \emph{mutatis mutandis}, similar results as derived in~\cite[App.~C]{BergBrei21} hold.
\end{remark}

We may now infer the following properties of a mild solution of~\eqref{eq:FPE-OU-dist} in the case $d=0$. First we recall the eigenvectors $v_k$ of $\Gamma$ and define the orthogonal matrix {
\begin{equation}\label{eq:VRLambda}
\begin{aligned}
  V &:= [v_1,\ldots,v_n]\in\R^{n\times n},\quad \text{and}\\
  \Lambda &:= \diag(\lambda_1,\ldots,\lambda_n),\quad R:= \diag(\sqrt{\lambda_1/2},\ldots,\sqrt{\lambda_n/2}).
\end{aligned}
\end{equation}
}

\begin{proposition}\label{Prop:sln-prop}
Use the assumptions from Definition~\ref{Def:solution}, assume that $d=0$ and let~$p$ be a mild solution of~\eqref{eq:FPE-OU-dist} on $[0,t_1]$. Then the following statements are true:
\begin{enumerate}
  \item $\int_{\R^n} p(t,x) {\rm d}x = \int_{\R^n} p_0(x) {\rm d}x$ for all $t\in [0,t_1]$.
  \item If $p_0(x)\ge 0$ for almost all $x\in\R^n$, then $p(t,x)\ge 0$ for all $t\in [0,t_1]$ and almost all $x\in\R^n$.
  \item {Recall~\eqref{eq:VRLambda}.} If $\int_{\R^n} p_0(x) {\rm d}x = 1$, then for $y:[0,t_1]\to\R^n$ as in~\eqref{eq:output} there exists $K\in\R^{n\times n}$, which is independent of~$t_1$, such that, for all $t\in[0,t_1]$,
\begin{align*}
    {\rm Cov}(t) = \int_{\R^n} \big(x-y(t)\big)\big(x-y(t)\big)^\top p(t,x){\rm d} x
= \tfrac14 V R^{-1} \big( e^{-c\Lambda t} K e^{-c\Lambda t} + 2 I\big) R^{-1} V^\top.
\end{align*}
If~$p$ is even a mild solution of~\eqref{eq:FPE-OU-dist} on~$\R_{\ge 0}$, then $\lim_{t\to\infty} {\rm Cov}(t) = c \Gamma^{-1}$.
\end{enumerate}
\end{proposition}
\begin{proof}
We show (i). By~\cite[Rem.~4.1.2]{TucsWeis09} the mild solution~$p$ admits the representation
\[
    \langle p(t) - p_0, v\rangle_\fH = \int_0^t \langle p(s), \fA v\rangle_\fH + \langle  \fB_1 F\big(p(s),u(s)\big), v\rangle_\fH {\rm d}s,\quad v\in\fV,
\]
where we have used that $\cA$ is self-adjoint by Proposition~\ref{Prop:eigenval-frakA}. Let $v = e^{-\phi}$, then $\fA v = 0$ by Proposition~\ref{Prop:eigenval-frakA} and
\begin{align*}
  \langle  \fB_1 F\big(p(s),u(s)\big), v\rangle_\fH &= - \int_{\R^n} \div\big( p(s,x) g(u(s))\big) {\rm d}x = -\lim_{r\to\infty} \int_{S_r}  p(s,x) g(u(s)) \cdot {\vec{n}}\,{\rm d}S = 0
\end{align*}
by a combination of Proposition~\ref{Prop:eigenval-A}\,(ii) and Lemma~\ref{Lem:sln-weak}, where we have used that $e^{\phi} p(s) g(u(s)) \in V^n$ for all $s\in[0,t_1]$. This proves the claim.

We show (ii). First we define the positive and negative part of~$p$ in the usual way by
\[
    p^+(t,x) := \max\{p(t,x), 0\},\quad p^-(t,x):=\max\{-p(t,x),0\}
\]
for $(t,x)\in[0,t_1]\times\R^n$. Since~$p$ is a mild solution and $\|p^\pm(t)\|_\fH \le \|p(t)\|_\fH$ for all $t\in[0,t_1]$, we have $p^\pm\in C([0,t_1];\fH)$. Define $\tilde H_\alpha := c_\alpha^{-1} H_\alpha$, where $c_\alpha = \|e^{-\phi} H_\alpha\|_\fH$ for $\alpha\in (\N_0)^n$. Then $w_\alpha := e^{-\phi} \tilde H_\alpha$ constitutes an orthonormal basis in $\fH$ and hence we have that
\[
    p^-(t) = \sum_{\alpha \in (\N_0)^n} \beta_\alpha(t) w_\alpha,\quad \beta_\alpha(t) = \langle p^-(t), w_\alpha\rangle_\fH,\quad t\in[0,t_1].
\]
Fix $k\in\N$, denote $|\alpha| = \alpha_1 + \ldots + \alpha_n$ for $\alpha\in (\N_0)^n$, and define $p_k^-(t) := \sum_{|\alpha|\le k} \beta_\alpha(t) w_\alpha$ for $t\in[0,t_1]$. Clearly $\ddt p^-(t) = \mathds{1}_{\{p<0\}} \dot p(t)$ and $\tfrac{\partial}{\partial x_i} p^-(t) = \mathds{1}_{\{p<0\}} \tfrac{\partial p}{\partial x_i}(t)$ for almost all $t\in [0,t_1]$, cf.\ e.g.~\cite[Thm.~2.8]{Chip00}. Hence, we have, {recalling~\eqref{eq:H-lambda-alpha}},
\begin{align*}
  \dot \beta_\alpha(t) & = \langle \dot p(t), \mathds{1}_{\{p<0\}} w_\alpha\rangle_\fH\stackrel{\eqref{eq:weak-repr}}{=} -c\langle p(t), {\fA \big( \mathds{1}_{\{p<0\}} w_\alpha\big)} \rangle_\fH + \langle\fB_1 F\big(p(t),u(t)\big), \mathds{1}_{\{p<0\}} w_\alpha\rangle_\fH\\
  &= -c\lambda_\alpha \langle {p^-(t)}, w_\alpha \rangle_\fH - \int_{\R^n} \div\big( p^-(t,x) g(u(t))\big) e^{\phi(x)} w_\alpha(x) {\rm d}x\\
  &\stackrel{(*)}{=}  -c\lambda_\alpha \beta_\alpha(t) + \int_{\R^n} p^-(t,x) g(u(t)) \nabla \big(e^{\phi(x)} w_\alpha(x)\big) {\rm d}x  \\
  &=  -c\lambda_\alpha \beta_\alpha(t) + \langle p^-(t) g(u(t)), e^{-\phi} \nabla (e^{\phi} w_\alpha)\rangle_{\fH^n}
\end{align*}
for almost all $t\in [0,t_1]$ and all $\alpha\in (\N_0)^n$, where $(*)$ follows from Proposition~\ref{Prop:eigenval-A}\,(ii) and Lemma~\ref{Lem:sln-weak} upon observing that $\|p^-\|_{L^1(0,t_1;\fV)} \le \|p\|_{L^1(0,t_1;\fV)}$ and hence $p^-\in L^1(0,t_1;\fV)$. {Further observe that $\fa\big(p_k^-(t),p_k^-(t)\big) = \langle p_k^-(t), \fA p_k^-(t) \rangle_\fH =  \sum_{|\alpha|\le k} \lambda_\alpha \beta_\alpha(t)^2$ by definition of~$w_\alpha$.} Therefore, we obtain, invoking Parseval's identity,
\begin{align*}
     &\tfrac12 \ddt \|p_k^-(t)\|_\fH^2 = \sum_{|\alpha|\le k} \beta_\alpha(t) \dot \beta_\alpha(t) = \sum_{|\alpha|\le k} \left( -c\lambda_\alpha \beta_\alpha(t)^2 \!+\! \langle p^-(t) g(u(t)), e^{-\phi} \nabla (e^{\phi} \beta_\alpha(t) w_\alpha)\rangle_{\fH^n}\right)\\
     &=-c \fa\big(p_k^-(t),p_k^-(t)\big) + \langle p^-(t) g(u(t)), e^{-\phi} \nabla (e^{\phi} p_k^-(t))\rangle_{\fH^n}\\
     &\le -c \fa\big(p_k^-(t),p_k^-(t)\big) + \|g(u)\|_{L^\infty(0,t_1;\R^n)} \|p^-(t)\|_\fH \|e^{-\phi} \nabla (e^{\phi} p_k^-(t))\|_{\fH^n}\\
     &\le -c \fa\big(p_k^-(t),p_k^-(t)\big) + \frac{1}{2c} \|g(u)\|_{L^\infty(0,t_1;\R^n)}^2 \|p^-(t)\|_\fH^2 + \frac{c}{2} \|e^{-\phi} \nabla (e^{\phi} p_k^-(t))\|_{\fH^n}^2 \le  \tfrac{D}{2} \|p^-(t)\|_\fH^2
\end{align*}
for almost all $t\in [0,t_1]$, where $D:= \frac{1}{c} \|g(u)\|_{L^\infty(0,t_1;\R^n)}^2$ and we have used that $\fa\big(p_k^-(t),p_k^-(t)\big) = \|e^{-\phi} \nabla (e^{\phi} p_k^-(t)\|_{\fH^n}^2$. Since
\[
    \|p_k^-(t)\|_\fH^2 = \sum_{|\alpha|\le k} \beta_\alpha(t)^2 \le \sum_{\alpha\in (\N_0)^n} \beta_\alpha(t)^2 = \|p^-(t)\|_\fH^2
\]
by Parseval's identity we find that $\eps_k(t) := \|p^-(t)\|_\fH^2 - \|p_k^-(t)\|_\fH^2 \ge 0$ and satisfies
\[
   \lim_{k\to\infty} \sup_{t\in[0,t_1]} \eps_k(t) = 0.
\]
Hence $\ddt \|p_k^-(t)\|_\fH^2 \le D \|p^-_k(t)\|_\fH^2 + D \eps_k(t)$, which implies
\[
    \|p^-_k(t)\|_\fH^2 \le e^{D t} \|p^-_k(0)\|_\fH^2 + \int_0^t D e^{D(t-s)} \eps_k(s) {\rm d}s \le e^{D t} \|p^-(0)\|_\fH^2 + e^{D t} \sup_{s\in[0,t_1]} \eps_k(s)
\]
for all $t\in[0,t_1]$ by Gr\"onwall's lemma. Since
\[
    p^-(0,x) = \max\{-p(0,x),0\} = \max\{-p_0(x),0\} = 0
\]
for almost all $x\in\R^n$, it follows that $\lim_{k\to\infty} \|p^-_k(t)\|_\fH^2 = 0$ for all $t\in[0,t_1]$, thus $p^-(t) = 0\in\fH$ and the claim is shown.

We show (iii). {Recall the definition of $u_1,\ldots,u_n$ from Section~\ref{Sec:FPO}.} Let $k,l\in\{1,\ldots,n\}$ with $k\neq l$ and define, for $t\in [0,t_1]$ and $x\in\R^n$,
\begin{align*}
  z_k^1(x)&:= e^{-\phi(x)} H_1(u_k^\top x),\quad z_{k,l}(x):= e^{-\phi(x)} H_1(u_k^\top x) H_1(u_l^\top x),\quad z_k^2(x):= e^{-\phi(x)} H_2(u_k^\top x),\\
  \mu_k^1(t)&:= \langle p(t), z_k^1 \rangle_\fH,\quad \mu_{k,l}(t):= \langle p(t), z_{k,l}\rangle_\fH,\quad \mu_k^2(t):= \langle p(t), z_k^2\rangle_\fH.
\end{align*}
Note that $\fA z_k^1 = \lambda_k z_k^1$, $\fA z_{k,l} = (\lambda_k + \lambda_l) z_{k,l}$ and $\fA z_k^2 = 2 \lambda_k z_k^2$ by Propostion~\ref{Prop:eigenval-frakA}. Then it follows from Lemma~\ref{Lem:sln-weak} that
\begin{align*}
    &\dot \mu_k^1(t) = -c \langle p(t), \fA z_k^1\rangle_\fH \!+\! \langle \fB_1 F\big(p(t),u(t)\big),z_k^1\rangle_{\fH} = -c\lambda_k \langle p(t), z_k^1\rangle_\fH \!-\!  \langle \div\big(p(t)g(u(t))\big), z_k^1\rangle_\fH \\
    &\stackrel{\rm Prop.\,\ref{Prop:eigenval-A}\,(ii)}{=}  -c\lambda_k \mu_k^1(t) + \int_{\R^n} p(t,x) g(u(t))^\top \nabla \left(e^{\phi(x)} z_k^1(x)\right) {\rm d}x \stackrel{\eqref{eq:nabla-Halpha}}{=} -c\lambda_k \mu_k^1(t) + 2 u_k^\top g(u(t))
\end{align*}
for almost all $t\in[0,t_1]$ and $k=1,\ldots,n$, where we have used that {$\int_{\R^n} p(t,x) {\rm d}x = \int_{\R^n} p_0(x) {\rm d}x = 1$ by~(i) and the assumption.} Analogously, we derive that
\begin{align*}
  \dot \mu_{k,l}(t) &= -c(\lambda_k+\lambda_l) \mu_{k,l}(t) + 2 \big(\mu_k^1(t) u_l + \mu_l^1(t) u_k\big)^\top g(u(t)),\\
  \dot \mu_k^2(t) &= -2c\lambda_k \mu_k^2(t)+ 4 \mu_k^1(t) u_k^\top g(u(t)).
\end{align*}
Now, {recall~\eqref{eq:VRLambda} and} let $F\in\R^{n\times n}$ be such that $I = [u_1,\ldots, u_n] F^\top = V R F^\top$. Then, for all $i,j = 1,\ldots,n$ and $x\in\R^n$, we have
\begin{align*}
  x_i x_j &=  (x^\top e_i) (x^\top e_j) = \left(\sum_{k=1}^n  F_{i,k} x^\top u_{k}\right)  \left(\sum_{l=1}^n  F_{j,l} x^\top u_{l}\right) \\
  &= \sum_{k=1}^n \sum_{l\neq k}  F_{i,k} F_{j,l} (x^\top u_{k}) (x^\top u_l) + \sum_{k=1}^n  F_{i,k} F_{j,k} (x^\top u_{k})^2 \\
  &=\frac14 e^{\phi(x)} \left(\sum_{k=1}^n \sum_{l\neq k}  F_{i,k} F_{j,l} z_{k,l}(x) +  \sum_{k=1}^n  F_{i,k} F_{j,k} \big(z_k^2(x) + 2\big) \right),
\end{align*}
by which
\[
    \int_{\R^n} x_i x_j p(t,x) {\rm d}x = \frac14 \sum_{k=1}^n \sum_{l\neq k}  F_{i,k} F_{j,l} \mu_{k,l}(t) +  \frac14 \sum_{k=1}^n  F_{i,k} F_{j,k} \big(\mu_k^2(t) + 2\big)
\]
for all $t\in[0,t_1]$. Furthermore, observe that
\[
    y_i(t) = \int_{\R^n} x_i p(t,x){\rm d}x = \frac12 \sum_{k=1}^n  F_{i,k} \mu_k^1(t),\quad i=1,\ldots,n,
\]
and define
\begin{align*}
    \big(M_A(t)\big)_{k,l} &:= \begin{cases} \mu_{k,l}(t), & k\neq l,\\ 0, & k=l,\end{cases} \ \text{for $k,l=1,\ldots,n$},\quad M_B(t) := \diag\big(\mu_1^2(t),\ldots,\mu_n^2(t)\big),\\
    \mu^1(t) &:= \big(\mu_1^1(t),\ldots,\mu_n^1(t)\big)^\top
\end{align*}
for $t\in[0,t_1]$. Then, the covariance matrix admits the representation
\begin{align*}
  {\rm Cov}(t) &= \int_{\R^n} \big(x-y(t)\big)\big(x-y(t)\big)^\top p(t,x){\rm d} x = \int_{\R^n} x x^\top p(t,x){\rm d} x - y(t)y(t)^\top \\
  &= \frac14 F \left(M_A(t) + M_B(t) + 2I - \mu^1(t) \mu^1(t)^\top \right) F^\top.
\end{align*}
We set $P(t) := M_A(t) + M_B(t) - \mu^1(t) \mu^1(t)^\top$ and {by using the equations derived above and accordingly rearranging the terms we may} compute the derivative as
\begin{align*}
  \dot P(t) &= -c (\Lambda M_A(t) + M_A(t) \Lambda) + 2 \big(\mu^1(t) g(u(t))^\top VR + (VR)^\top g(u(t)) \mu^1(t)^\top\big)\\
  &\quad - 2c \Lambda M_B(t) \!-\! \big( 2 (VR)^\top g(u(t)) \!-\! c\Lambda \mu^1(t)\big) \mu^1(t)^\top \!-\! \mu^1(t) \big(  2 (VR)^\top g(u(t)) \!-\! c\Lambda \mu^1(t)\big)^\top\\
  &= - c(\Lambda P(t) + P(t) \Lambda),
\end{align*}
by which
\[
    P(t) = e^{-c\Lambda t} P(0) e^{-c\Lambda t},\quad t\in[0,t_1],
\]
and hence, invoking $F^\top = (VR)^{-1}$, the claim is shown.

The last statement follows from $e^{-c\Lambda t} \to 0$ for $t\to \infty$ and the observation that
\[
    \tfrac12 V R^{-2} V^\top = V \Lambda^{-1} V^\top = c \Gamma^{-1}.
\]
\end{proof}

Finally, we show boundedness of the mild solution on $\R_{\ge 0}$ for bounded inputs and disturbances satisfying condition~\eqref{eq:cond-dist}.

\begin{lemma}\label{lem:bounded-sln}
Use the assumptions from Definition~\ref{Def:solution}, further assume that $u\in L^\infty(0,\infty;\R^n)$ and~\eqref{eq:cond-dist} holds, and let~$p$ be a mild solution of~\eqref{eq:FPE-OU-dist} on $\R_{\ge 0}$. Then $p\in L^\infty(0,\infty;\fH)$.
\end{lemma}
\begin{proof} Recall the orthonormal basis $w_\alpha = c_\alpha^{-1} e^{-\phi} H_\alpha$ of~$\fH$ and the constants $c_\alpha = \|e^{-\phi} H_\alpha\|_\fH$ from the proof of Proposition~\ref{Prop:sln-prop}. Then we have that
\[
    p(t) = \sum_{\alpha\in (\N_0)^n} \beta_\alpha(t) w_\alpha,\quad \beta_\alpha(t) = \langle p(t), w_\alpha\rangle_\fH,\quad t\ge 0.
\]
Furthermore, for $\alpha\in (\N_0)^n$ and $i=1,\ldots,n$ we define the multi-index $\alpha^{-i}$ by
\[
    \alpha^{-i}_j := \begin{cases} \alpha_j, & j\neq i,\\ \alpha_i-1, & j=i,\end{cases}\quad j=1,\ldots,n,
\]
which may have an entry which is $-1$. If $\alpha^{-i}_i = -1$, then we define $\beta_{\alpha^{-i}}(t) := 0$. Furthermore, by the properties of the Hermite polynomials, we have that
\begin{equation}\label{eq:c-alpha}
    2 \alpha_i c_{\alpha^{-i}}^2 = c_\alpha^2.
\end{equation}
Now, fix $k\in\N$ and define 
\[
    {p_k(t) := \sum_{|\alpha|\le k} \beta_\alpha(t) w_\alpha.}
\]
Then, similar as in the proof of Proposition~\ref{Prop:sln-prop}, we may compute that
\begin{align*}
    \dot \beta_\alpha(t) = -c\lambda_\alpha \beta_\alpha(t) + \langle p(t) g(u(t)), e^{-\phi}\nabla(e^{\phi} w_\alpha)\rangle_{{\fH^n}} + \langle d(t), w_\alpha\rangle_\fH.
\end{align*}
Furthermore, we find
\begin{align*}
 &e^{-\phi(x)} \nabla\big(e^{\phi(x)} w_\alpha(x)\big) = c_\alpha^{-1} e^{-\phi(x)} \nabla H_\alpha(x) \\
 &\stackrel{\eqref{eq:nabla-Halpha}}{=} 2 c_\alpha^{-1} e^{-\phi(x)} \sum_{j=1}^n \left( \prod_{i\neq j} H_{\alpha_i}(u_i^\top x)\right) \alpha_j H_{\alpha_j-1}(u_j^\top x) u_{j}  
 = 2 c_\alpha^{-1} \sum_{j=1}^n \alpha_j e^{-\phi(x)} H_{\alpha^{-j}}(x) u_{j} \\
 & \stackrel{\eqref{eq:c-alpha}}{=}  2 c_\alpha^{-1} \sum_{j=1}^n \alpha_j \frac{c_\alpha}{\sqrt{2\alpha_j}} w_{\alpha^{-j}}(x) u_{j} 
= \sum_{j=1}^n \sqrt{2\alpha_j}  w_{\alpha^{-j}}(x) u_{j},
\end{align*}
which gives that
\begin{equation}\label{eq:beta-dot}
 \dot \beta_\alpha(t) = -c\lambda_\alpha \beta_\alpha(t) + \sum_{j=1}^n \sqrt{2\alpha_j}  g(u(t))^\top u_{j}  \beta_{\alpha^{-j}}(t) + \langle d(t), w_\alpha\rangle_\fH.
\end{equation}
By Parseval's identity we have that
\[
    \|p_k(t)\|_\fH^2 = \sum_{|\alpha|\le k} \beta_\alpha(t)^2 \le \sum_{\alpha\in(\N_0)^n} \beta_\alpha(t)^2 = \|p(t)\|_\fH^2,\quad t\ge 0,
\]
and hence, using the notation $\|g(u)\|_\infty := \|g(u)\|_{L^\infty(0,\infty;\R^n)}$ and $\|d\|_\infty := \|d\|_{L^\infty(0,\infty;\fH)}$ {as well as recalling that $\|u_j\|_{\R^n} = \sqrt{\lambda_j/2}$, we find that}
\begin{align*}
  &\tfrac12 \ddt \|p_k(t)\|_\fH^2 = \sum_{|\alpha|\le k} \left( -c\lambda_\alpha \beta_\alpha(t)^2 + \sum_{j=1}^n \sqrt{2\alpha_j}  g(u(t))^\top u_{j} \beta_\alpha(t) \beta_{\alpha^{-j}}(t)\right) + \langle d(t), p_k(t)\rangle_\fH\\
  &\le \sum_{|\alpha|\le k} \sum_{j=1}^n \left( -c\lambda_j \alpha_j \beta_\alpha(t)^2 +  \sqrt{2\alpha_j} \|g(u)\|_\infty \|u_j\|_{\R^n} \beta_\alpha(t) \beta_{\alpha^{-j}}(t)\right) + \|d\|_\infty \|p_k(t)\|_\fH\\
  &\le \sum_{|\alpha|\le k} \sum_{j=1}^n \left( -c\lambda_j \alpha_j \beta_\alpha(t)^2 +  \tfrac12 \sqrt{\lambda_j \alpha_j} \|g(u)\|_\infty \big( \beta_\alpha(t)^2 +  \beta_{\alpha^{-j}}(t)^2\big)\right) + \|d\|_\infty \|p_k(t)\|_\fH\\
  &= \sum_{|\alpha|\le k} \sum_{j=1}^n \left( -c\lambda_j \alpha_j  +  \tfrac12  \|g(u)\|_\infty \left(\sqrt{\lambda_j \alpha_j} + \sqrt{\lambda_j (\alpha_j+1)}\right)\right)\beta_\alpha(t)^2\\
  &\quad -  \sum_{j=1}^n \sum_{|\alpha|\le k,\, \alpha_j=k} \tfrac12  \|g(u)\|_\infty \sqrt{\lambda_j (k+1)} \beta_\alpha(t)^2 + \|d\|_\infty \|p_k(t)\|_\fH\\
  &\le \sum_{|\alpha|\le k} \sum_{j=1}^n \left( -c\lambda_j \alpha_j  +  \|g(u)\|_\infty \left(\sqrt{\lambda_j \alpha_j} + \tfrac12 \sqrt{\lambda_j}\right)\right)\beta_\alpha(t)^2 + \|d\|_\infty \|p_k(t)\|_\fH.
\end{align*}
Define
\[
    \eta:(\N_0)^n \to \R,\ \alpha \mapsto \sum_{j=1}^n \left( -c\lambda_j \alpha_j  +  \|g(u)\|_\infty \left(\sqrt{\lambda_j \alpha_j} + \tfrac12 \sqrt{\lambda_j}\right)\right),
\]
then it is clear that there exists $k_0\in\N$ such that $\eta(\alpha) < 0$ for all $\alpha\in(\N_0)^n$ with $|\alpha|>k_0$. W.l.o.g.\ we may choose $k_0$ large enough so that
\[
    \frac{\partial \eta}{\partial \alpha_j}(\alpha) = -c \lambda_j + \frac{\|g(u)\|_\infty \lambda_j}{2 \sqrt{\lambda_j \alpha_j}} < 0
\]
for all $\alpha\in(\N_0)^n$ with $\alpha_j>k_0$ and all $j=1,\ldots,n$. Then we have that
\begin{align*}
    \eta_0 &:= \sup\setdef{\eta(\alpha)}{\alpha\in(\N_0)^n, |\alpha|>k_0} \\
    &\le \max\setdef{\eta(\alpha)}{\alpha\in(\N_0)^n, |\alpha|>k_0,\ \forall\, j=1,\ldots,n:\ \alpha_j \le k_0+1} < 0.
\end{align*}
With $\kappa_1(t):= \sum_{|\alpha|\le k_0} \eta(\alpha) \beta_\alpha(t)^2$ we obtain, for all $k>k_0$,
\begin{align*}
  \tfrac12 \ddt \|p_k(t)\|_\fH^2 &\le \kappa_1(t) +  \sum_{k_0<|\alpha|\le k} \eta(\alpha) \beta_\alpha(t)^2 + \|d\|_\infty \|p_k(t)\|_\fH\\
  &\le \eta_0\left( \sum_{|\alpha|\le k} \beta_\alpha(t)^2 - \sum_{|\alpha|\le k_0} \beta_\alpha(t)^2\right) + \kappa_1(t) + \frac{\|d\|_\infty^2}{2|\eta_0|} + \frac{|\eta_0|}{2}\|p_k(t)\|_\fH^2\\
  &\le \frac{\eta_0}{2} \|p_k(t)\|_\fH^2 + \kappa_1(t) + \kappa_2(t) + \frac{\|d\|_\infty^2}{2|\eta_0|}
\end{align*}
for all $t\ge 0$, where $\kappa_2(t):= - \eta_0 \sum_{|\alpha|\le k_0} \beta_\alpha(t)^2$. To conclude the proof we show that $\beta_\alpha(\cdot)$ is bounded for all $\alpha\in (\N_0)^n$. To this end, observe that for $\alpha = 0 = (0,\ldots,0)$ we have
\[
    \dot \beta_0(t) = - c\underbrace{\lambda_0}_{=0} \beta_0(t)+ \langle p(t) g(u(t)), e^{-\phi}\underbrace{\nabla(e^{\phi} w_0)}_{=\nabla H_0 = 0}\rangle_\fH + \underbrace{\langle d(t), w_0\rangle_\fH}_{\stackrel{\eqref{eq:cond-dist}}{=} 0} = 0,
\]
thus $\beta_0(t) = \langle p_0, w_0\rangle_\fH$ for all $t\ge 0$. Then a simple induction based on~\eqref{eq:beta-dot} and invoking boundedness of~$d$ shows that $\beta_\alpha\in L^\infty(0,\infty;\R)$ for all $\alpha\in(\N_0)^n$. Therefore, boundedness of~$\kappa_1$ and~$\kappa_2$ follows, which yields that
\[
    \ddt \|p_k(t)\|_\fH^2 \le \eta_0 \|p_k(t)\|_\fH^2 + M,\quad M:= 2\|\kappa_1 + \kappa_2\|_{L^\infty(0,\infty;\R)} + \frac{\|d\|_\infty^2}{|\eta_0|}
\]
for all $t\ge 0$. Then Gr\"{o}nwall's lemma implies that, for all $k>k_0$,
\begin{align*}
  \forall\, t\ge 0:\ \|p_k(t)\|_\fH^2 \le \|p_k(0)\|_\fH^2 e^{\eta_0 t} + \frac{M}{|\eta_0|} \le \|p_0\|_\fH^2 + \frac{M}{|\eta_0|} =: \tilde M,
\end{align*}
by which $\|p(t)\|_\fH^2 = \lim_{k\to\infty} \|p_k(t)\|_\fH^2 \le \tilde M$ for all $t\ge 0$.
\end{proof}

\section{A simple feedforward controller}\label{Sec:NonRobCon}

In this section we present a very simple, yet effective feedforward control strategy. We stress that the presented control law does not achieve the control objective -- it is not robust and does not guarantee error evolution within the prescribed performance funnel. Nevertheless, we will show that it guarantees fast (exponential) convergence of the tracking error to zero, provided the system parameters are known, the nonlinearity~$g$ is the identity, no disturbances are present and the derivative of the reference signal is available to the controller. For $\Gamma = \Gamma^\top>0$ as in~\eqref{eq:FPE-OU} and reference signal $y_{\rm ref}\in W^{1,\infty}(\R_{\ge 0};\R^n)$ the controller is given by
\begin{equation}\label{eq:non-rob-con}
    u(t) = \dot y_{\rm ref}(t) + \Gamma y_{\rm ref}(t).
\end{equation}
Note that~\eqref{eq:non-rob-con} is not a feedback controller, it is completely determined by~$y_{\rm ref}$. We show that~\eqref{eq:FPE-OU-dist} with~\eqref{eq:non-rob-con} admits a solution.

\begin{proposition}\label{Prop:non-rob-con}
Use the assumptions from Definition~\ref{Def:solution} such that $g = {\rm id}_{\R^n}$ and $d=0$, and let $y_{\rm ref}\in W^{1,\infty}(\R_{\ge 0};\R^n)$. Then there exists a unique mild solution~$p$ of~\eqref{eq:FPE-OU-dist} with~\eqref{eq:non-rob-con} on~$\R_{\ge 0}$ such that
\begin{enumerate}
  \item $p\in L^q_{\loc}(0,\infty;\fV) \cap W^{1,q}_{\loc}(0,\infty;\fH_{-{\frac12}})\cap L^\infty(0,\infty;\fH)$ for all $1\le q<2$  and
  \item for the output~$y$ defined in~\eqref{eq:output} and $P_0 := \int_{\R^n} p_0(x){\rm d} x$ we have that
  \[
    \forall\, t\ge 0:\ y(t) = P_0 y_{\rm ref}(t) + e^{-\Gamma t} \big(y(0) - P_0 y_{\rm ref}(0)\big).
  \]
\end{enumerate}
Furthermore,~$p$ exhibits the properties derived in Lemma~\ref{Lem:sln-weak} and Proposition~\ref{Prop:sln-prop}.
\end{proposition}
\begin{proof} We show existence and uniqueness of a mild solution. Let $t_1>0$ be arbitrary and define $\tilde u := \mathds{1}_{[0,t_1]} u$ for~$u$ as in~\eqref{eq:non-rob-con}. Then, since  $y_{\rm ref}\in W^{1,\infty}(\R_{\ge 0};\R^n)$, we have $\tilde u \in L^2(0,\infty;\R^n)$. Furthermore, $\fB_1$ is $L^2$-admissible by Lemma~\ref{Lem:admiss}, $g = {\rm id}_{\R^n}$ and {hence~\cite[Lem.~2.8]{HosfJaco20} together with Lemma~\ref{lem:bounded-sln} (applied, \textit{mutatis mutandis}, to the interval~$[0,t_1]$ instead of $\R_{\ge 0}$)}  yields the existence of a unique mild solution~$\tilde p_{t_1}$ of~\eqref{eq:FPE-OU-dist} with input~$\tilde u$ on~$[0,t_1]$. Define $p:\R_{\ge 0}\to\fH$ by $p|_{[0,t_1]} := \tilde p_{t_1}$ for any $t_1>0$, which is well-defined by uniqueness of~$\tilde p_{t_1}$. Then~$p$ is the unique mild solution of~\eqref{eq:FPE-OU-dist} with~\eqref{eq:non-rob-con} on~$\R_{\ge 0}$.

Statement (i) follows from Lemmas~\ref{Lem:sln-weak} and~\ref{lem:bounded-sln} together with $u\in L^\infty(0,\infty;\R^n)$ and $d=0$. It remains to show~(ii). Using the notation from the proof of Proposition~\ref{Prop:sln-prop} we find that $y(t) = \tfrac12 F \mu^1(t)$ and, invoking $g = {\rm id}_{\R^n}$,
\[
    \dot y(t) = -\frac{c}{2} F \Lambda \mu^1(t) + P_0 F (VR)^\top u(t) = -c F \Lambda F^{-1} y(t) + P_0 u(t), \quad t\ge 0,
\]
where we have used that $F^\top = (VR)^{-1}$. Recalling $c F \Lambda F^{-1} = c V \Lambda V^\top = \Gamma$, together with~\eqref{eq:non-rob-con} we now obtain that $\ddt \big( y(t) - P_0y_{\rm ref}(t)\big) = -\Gamma  \big( y(t) - P_0y_{\rm ref}(t)\big)$, from which the claim follows directly.
\end{proof}

We emphasize that the result of Proposition~\ref{Prop:non-rob-con} is independent of the initial value~$p_0\in\fH$. Moreover, if~$p_0$ satisfies $\int_{\R^n} p_0(x) {\rm d}x =1$, then the control~\eqref{eq:non-rob-con} achieves exponential convergence of the tracking error~$e(t) = y(t) -y_{\rm ref}(t)$ to zero for all initial probability densities. Furthermore, the mild solution~$p$ exhibits the properties derived in Proposition~\ref{Prop:sln-prop}; thus its mean value and covariance matrix exponentially converge to~$y_{\rm ref}$ and~$c\Gamma^{-1}$, resp.

Although the controller~\eqref{eq:non-rob-con} requires knowledge of~$\Gamma$ and~$\dot y_{\rm ref}$ and the absence of disturbances, its simplicity may justify its application in real-world examples. On the other hand, in the presence of uncertainties and disturbances, a feedback control strategy is more suitable, for which we refer to Section~\ref{Sec:FunCon}.

\section{Funnel control}\label{Sec:FunCon}

The controller that we propose in order to achieve the control objective formulated in Subsection~\ref{Ssec:ContrObj} is the funnel controller. It has the advantage that it is model-free, i.e., we may state the control law without any further information about the equation~\eqref{eq:FPE-OU}. Therefore, it is inherently robust and hence able to handle both uncertainties in the system parameters as well as disturbances in the PDE itself. In particular, we do not need any knowledge of the parameters~$c>0$,~$\Gamma\in\R^{n\times n}$ and $g\in C^1(\R^n;\R^n)$, or of the initial probability density~$p_0(\cdot)$. Furthermore, we seek robustness of the controller w.r.t.\ disturbances $d\in L^\infty(0,\infty;\fH)$ that satisfy the zero-mass condition~\eqref{eq:cond-dist}.

The proof of feasibility of funnel control strongly relies on showing that the output~\eqref{eq:output} corresponding to any mild solution of~\eqref{eq:FPE-OU-dist} satisfies the equation
\begin{equation}\label{eq:ouput-DGL}
    \dot y(t) = -\Gamma y(t) + P_0 g(u(t))+ \bar d(t),\ \ \text{where}\ \ P_0 = \int_{\R^n} p_0(x) {\rm d}x,\ \ \bar d(t) = \int_{\R^n} x d(t,x) {\rm d}x.
\end{equation}
Then this equation (under the funnel control feedback law stated below) may be solved separately and the resulting control input~$u$ may be inserted in~\eqref{eq:FPE-OU-dist}, which may be treated as an open-loop problem then for which~\cite{HosfJaco20} provides a solution. It can then be shown that this solution has the desired properties and the corresponding output generated via~\eqref{eq:output} equals~$y$ from~\eqref{eq:ouput-DGL}.

Utilizing the version of the funnel controller from~\cite{BergIlch20pp}, we only require the relative degree in order to state the appropriate control law. For finite dimensional systems we refer to~\cite{Isid95} for a definition of the relative degree; this notion can be extended to systems with infinite-dimensional internal dynamics, see e.g.~\cite{BergPuch20a}. However, for general infinite-dimensional systems a concept of relative degree is not available. Since the input appears explicitly in the equation~\eqref{eq:ouput-DGL} for~$\dot y$, this suggests that~\eqref{eq:FPE-OU-dist},~\eqref{eq:output} at least exhibits an input-output behavior similar to that of a relative degree one system. This justifies to investigate the application of the funnel controller
\begin{equation}\label{eq:fun-con}
u(t)=\big(N\circ \alpha\big)\big(\|w(t)\|_{\R^n}^2\big) w(t),\quad w(t) = \varphi(t)\big( y(t) - y_{\rm ref}(t)\big)
\end{equation}
to~\eqref{eq:FPE-OU-dist},~\eqref{eq:output}, where the funnel control design parameters are
\begin{equation}\label{eq:fun-con-design}
    \begin{cases}
     \ \varphi\in\Phi,\\
     \ \alpha\in C^1([0,1);[1,\infty))\ \ \text{a bijection},\\
     \ N\in C^1(\R_{\ge 0};\R)\ \ \text{a surjection};
    \end{cases}
\end{equation}
see~\cite{BergIlch20pp} for more details and explanations on the controller desgin. 
Typical choices for~$N$ and~$\alpha$ are $N(s) = s \cos s$ and $\alpha(s) = 1/(1-s)$.

For feasibility we seek to show that for any $y_{\rm ref}\in W^{1,\infty}(\R_{\ge 0};\R^n)$, a triple $(\varphi,\alpha,N)$ as in~\eqref{eq:fun-con-design}, a disturbance $d\in L^\infty(0,\infty;\fH)$ with~\eqref{eq:cond-dist} and any initial probability density~$p_0\in\fH$ such that $\varphi(0) \|e(0)\|_{\R^n} < 1$ we have that the closed-loop system consisting of~\eqref{eq:FPE-OU-dist},~\eqref{eq:output} and~\eqref{eq:fun-con} has a unique global and bounded mild solution~$p$ which satisfies the conditions~\eqref{eq:FPE-cond} and the tracking error~$e$ evolves uniformly within the performance funnel~$\cF_\varphi$ from~\eqref{eq:perf_funnel}.

Hence, even if a solution exists on a finite time interval $[0,t_1)$, it is not clear that it can be extended to a global solution. Moreover, the closed-loop system~\eqref{eq:FPE-OU-dist},~\eqref{eq:output} and~\eqref{eq:fun-con} is a time-varying and nonlinear PDE. This renders the solution of the above problem a challenging task.

Under the assumptions from Definition~\ref{Def:solution} and for $y_{\rm ref}\in W^{1,\infty}(\R_{\ge 0};\R^n)$ and a triple $(\varphi,\alpha,N)$ as in~\eqref{eq:fun-con-design}, we call $(p,u,y)$ a mild solution of~\eqref{eq:FPE-OU-dist},~\eqref{eq:output},~\eqref{eq:fun-con} on $[0,t_1]$, if $u,y\in C([0,t_1];\R^n)$ such that~\eqref{eq:output},~\eqref{eq:fun-con} hold for all $t\in [0,t_1]$ and~$p$ is a mild solution of~\eqref{eq:FPE-OU-dist} on $[0,t_1]$. A triple $(p,u,y)$ is called mild solution of~\eqref{eq:FPE-OU-dist},~\eqref{eq:output},~\eqref{eq:fun-con} on $\R_{\ge 0}$, if $(p,u,y)|_{[0,t_1]}$ is a mild solution of~\eqref{eq:FPE-OU-dist},~\eqref{eq:output},~\eqref{eq:fun-con} on $[0,t_1]$ for all $t_1>0$.

In the following main result of the present paper we prove feasibility of funnel control for the Fokker-Planck equation corresponding to the multi-dimensional Ornstein-Uhlenbeck process.

\begin{theorem}\label{Thm:fun-FPE}
Use the assumptions from Definition~\ref{Def:solution} (except for that on~$u$) and let $y_{\rm ref}\in W^{1,\infty}(\R_{\ge 0};\R^n)$, $(\varphi,\alpha,N)$ be a triple of funnel control design parameters as in~\eqref{eq:fun-con-design} and $E_0 := \int_{\R^n}  x p_0(x) {\rm d}x$, and assume that~$d$ satisfies~\eqref{eq:cond-dist},
\[
   P_0 = \int_{\R^n} p_0(x) {\rm d}x \neq 0\quad \text{and} \quad \varphi(0)\|E_0 - y_{\rm ref}(0)\|_{\R^n} < 1.
\]
Then there exists a unique mild solution $(p,u,y)$ of~\eqref{eq:FPE-OU-dist},~\eqref{eq:output},~\eqref{eq:fun-con} on $\R_{\ge 0}$ which satisfies
\begin{enumerate}
  \item $p\in L^q_{\loc}(0,\infty;\fV) \cap W^{1,q}_{\loc}(0,\infty;\fH_{-{\frac12}})\cap L^\infty(0,\infty;\fH)$ for all $1\le q<2$, $u\in C(\R_{\ge 0};\R^n) \cap L^{\infty}(\R_{\ge 0};\R^n)$, $y\in W^{1,\infty}(\R_{\ge 0};\R^n)$ and
  \item $\exists\,\eps\in(0,1)\ \forall\, t\ge 0:\ \varphi(t) \|e(t)\|_{\R^n} \le \eps$.
\end{enumerate}
Furthermore,~$p$ has the properties derived in Lemma~\ref{Lem:sln-weak} and, if $d=0$, Proposition~\ref{Prop:sln-prop}.
\end{theorem}
\begin{proof} \emph{Step 1}: Consider the equation~\eqref{eq:ouput-DGL} with initial condition $y(0)=E_0$ and observe that $P_0 \neq 0$ by assumption, $\|\bar d(t)\|_{\R^n}\le \kappa \|d(t)\|_\fH \le \kappa \|d\|_{L^\infty(0,\infty;\fH)}$ for some $\kappa>0$ and  almost all $t\ge 0$ and $\varphi(0)\|y(0) - y_{\rm ref}(0)\|_{\R^n} < 1$. Therefore, by property~\eqref{eq:HGP} of~$g$, existence of a solution to~\eqref{eq:ouput-DGL} under the control~\eqref{eq:fun-con} follows from~\cite[Thm.~1.8]{BergIlch20pp}, that is there exists a function $y\in C(\R_{\ge 0};\R)$ which is absolutely continuous on $[0,t_1]$ for all $t_1>0$ and satisfies $y(0) =E_0$ and~\eqref{eq:ouput-DGL} together with~\eqref{eq:fun-con} for almost all $t\ge 0$. Moreover, we have that $u\in C(\R_{\ge 0};\R^n) \cap L^{\infty}(\R_{\ge 0};\R^n)$ and $y\in W^{1,\infty}(\R_{\ge 0};\R^n)$ as well as the estimate in~(ii).

\emph{Step 2}: We show uniqueness of the solution $(u,y)$ of~\eqref{eq:ouput-DGL} under~\eqref{eq:fun-con} on~$\R_{\ge 0}$. Assume that~$(u^1,y^1)$ and~$(u^2,y^2)$ are two solutions of~\eqref{eq:ouput-DGL},~\eqref{eq:fun-con} on $\R_{\ge 0}$ with the same initial values $y^1(0) = E_0 = y^2(0)$. Then $y^i$ is the solution of the initial value problem
\begin{align*}
    \dot y^i(t) &= -\Gamma y^i(t) + P_0 g(u^i(t)) + \bar d(t),\quad u^ i(t)=\big(N\circ \alpha\big)\big(\|w^i(t)\|_{\R^n}^2\big) w^i(t),\\
     w^i(t) &= \varphi(t)\big( y^i(t) - y_{\rm ref}(t)\big),\quad y^i(0) = E_0.
\end{align*}
Since the right hand side of the ordinary differential equation above is measurable in~$t$ and locally Lipschitz continuous in~$y^i$ (since $g$, $N$ and $\alpha$ are continuously differentiable), its solution is unique, see e.g.~\cite[$\S$\,10,\,Thm.\,XX]{Walt98}. Since $y^1(0) = y^2(0)$ this implies that $y^1(t) = y^2(t)$ and hence also $u^1(t) = u^2(t)$ for all $t\ge 0$.

\emph{Step 3}: We show that there exists a unique mild solution $(p,u,y)$ of~\eqref{eq:FPE-OU-dist},~\eqref{eq:output},~\eqref{eq:fun-con} on $\R_{\ge 0}$, where~$u,y$ are defined in Step~1. The arguments are analogous to those in the proof of Proposition~\ref{Prop:non-rob-con}, additionally using $\tilde d := \mathds{1}_{[0,t_1]} d \in L^2(0,\infty;\fH)$, observing that $\fB_2$ is clearly $L^2$-admissible and that~$F$ satisfies
\begin{align*}
   &\|F(v,w)\|_{\fH^n} = \|v\|_\fH \|g(w)\|_{\R^n} \stackrel{\eqref{eq:HGP}}{\le} \bar g \|v\|_\fH \|w\|_{\R^n}\quad \text{and}\\
   &\|F(v_1,w)-F(v_2,w)\|_{\fH^n} = \|v_1-v_2\|_\fH \|g(w)\|_{\R^n} \stackrel{\eqref{eq:HGP}}{\le} \bar g \|v_1-v_2\|_\fH \|w\|_{\R^n}
\end{align*}
for all $v,v_1,v_2\in\fH$ and $w\in\R^n$.  It remains to show that~\eqref{eq:output},~\eqref{eq:fun-con} are satisfied for all $t \ge 0$. To this end, it suffices to observe that, recalling the findings from the proof of Proposition~\ref{Prop:non-rob-con}, we obtain that the output given in~\eqref{eq:output} satisfies the equation~\eqref{eq:ouput-DGL} with
\[
   \bar d(t) = \frac12 F \begin{pmatrix} \langle d(t), e^{-\phi} H_1(u_1^\top x)\rangle_\fH\\ \vdots \\ \langle d(t), e^{-\phi} H_1(u_n^\top x)\rangle_\fH\end{pmatrix} = F  (VR)^\top  \begin{pmatrix} \langle d(t), e^{-\phi} x_1 \rangle_\fH\\ \vdots \\ \langle d(t), e^{-\phi} x_n\rangle_\fH\end{pmatrix} = \int_{\R^n} x d(t,x) {\rm d}x
\]
with~$F,V,R\in\R^{n\times n}$ as in the proof of Proposition~\ref{Prop:sln-prop}. Finally, together with uniqueness of $u,y$ from Step~2, we obtain a unique mild solution~$p$.

\emph{Step 4}: The remaining assertion on~$p$ in~(i) follows from Lemmas~\ref{Lem:sln-weak} and~\ref{lem:bounded-sln}.
\end{proof}

\section{A numerical example}\label{Sec:Sim}

In this section, we illustrate the applicability of the funnel controller by means of a numerical example. We consider the one-dimensional case $n=1$ and simulate the evolution of a given initial probability density~$p_0$ under the Fokker-Planck equation~\eqref{eq:FPE-OU-dist} with the mean value as output~\eqref{eq:output} and under the influence of the controller~\eqref{eq:fun-con}. To show the universality of Theorem~\ref{Thm:fun-FPE} we consider an initial density that is in~$\fH$, but not in~$\fV$, namely a uniform distribution on $\big[-1,-\tfrac12\big] \cup \big[\tfrac14, \tfrac34\big]$ given by
\[
    p_0:\R\to\R,\ x\mapsto \begin{cases} 1, & -1 \le x \le -\tfrac12\ \vee \tfrac14 \le x \le \tfrac34,\\ 0, & \text{otherwise}\end{cases}\quad \in \fH \setminus\fV.
\]
The parameters in~\eqref{eq:FPE-OU} are chosen as $c=0.1$, $\Gamma = 1$ and $g = {\rm id}_{\R}$, the reference signal is $y_{\rm ref}(t) = \sin t$ and the funnel control design parameters are $\alpha(s) = 1/(1-s)$, $N(s) = s\cos s$ and $\varphi(t) = \left(2 e^{-2t} + 0.1\right)^{-1}$, which satisfy~\eqref{eq:fun-con-design}. As disturbance we consider
\[
    d:\R_{\ge 0}\times\R\to \R,\ (t,x)\mapsto 3\cos (4t)\, x e^{-3x^2},
\]
which clearly satisfies $d\in L^\infty(0,\infty;\fH)$ and condition~\eqref{eq:cond-dist}. Since $E_0 = \int_{-\infty}^\infty x p_0(x) {\rm d}x = -\frac18$ and $y_{\rm ref}(0) = 0$, it follows that $\varphi(0)|E_0 - y_{\rm ref}(0)| = \frac{5}{84} < 1$. Therefore, feasibility of funnel control, i.e., the application of~\eqref{eq:fun-con} to~\eqref{eq:FPE-OU-dist},~\eqref{eq:output}, is guaranteed by Theorem~\ref{Thm:fun-FPE}.

\begin{figure}[h!tb]
\begin{subfigure}{.48\linewidth}
\begin{center}
  \includegraphics[scale=0.49]{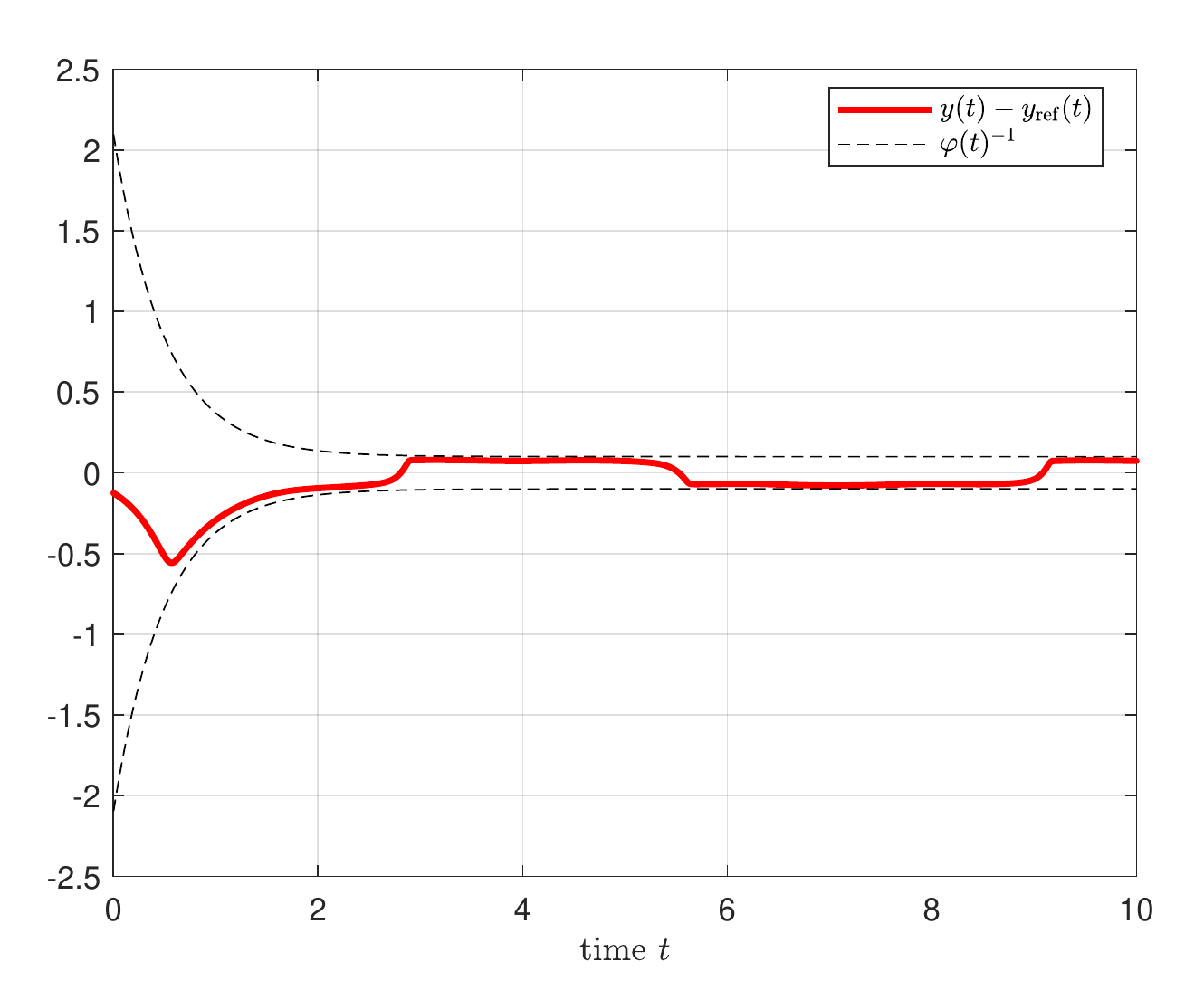}
  \caption{Tracking error and funnel boundary}
  \end{center}
\end{subfigure}\quad
\begin{subfigure}{.48\linewidth}
\begin{center}
  \includegraphics[scale=0.49]{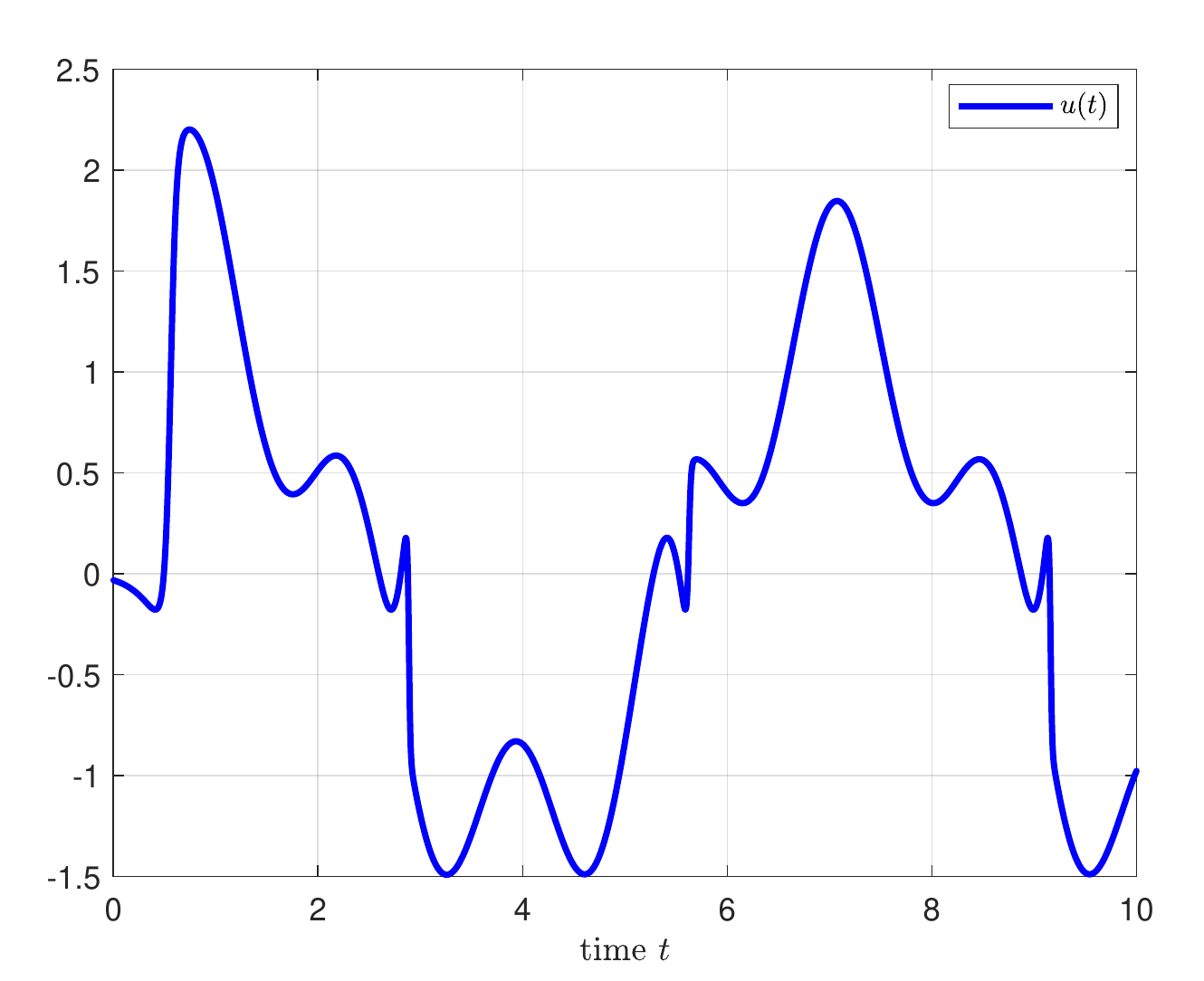}
  \caption{Input function}
  \end{center}
\end{subfigure}
\begin{subfigure}{.48\linewidth}
\begin{center}
  \includegraphics[scale=0.49]{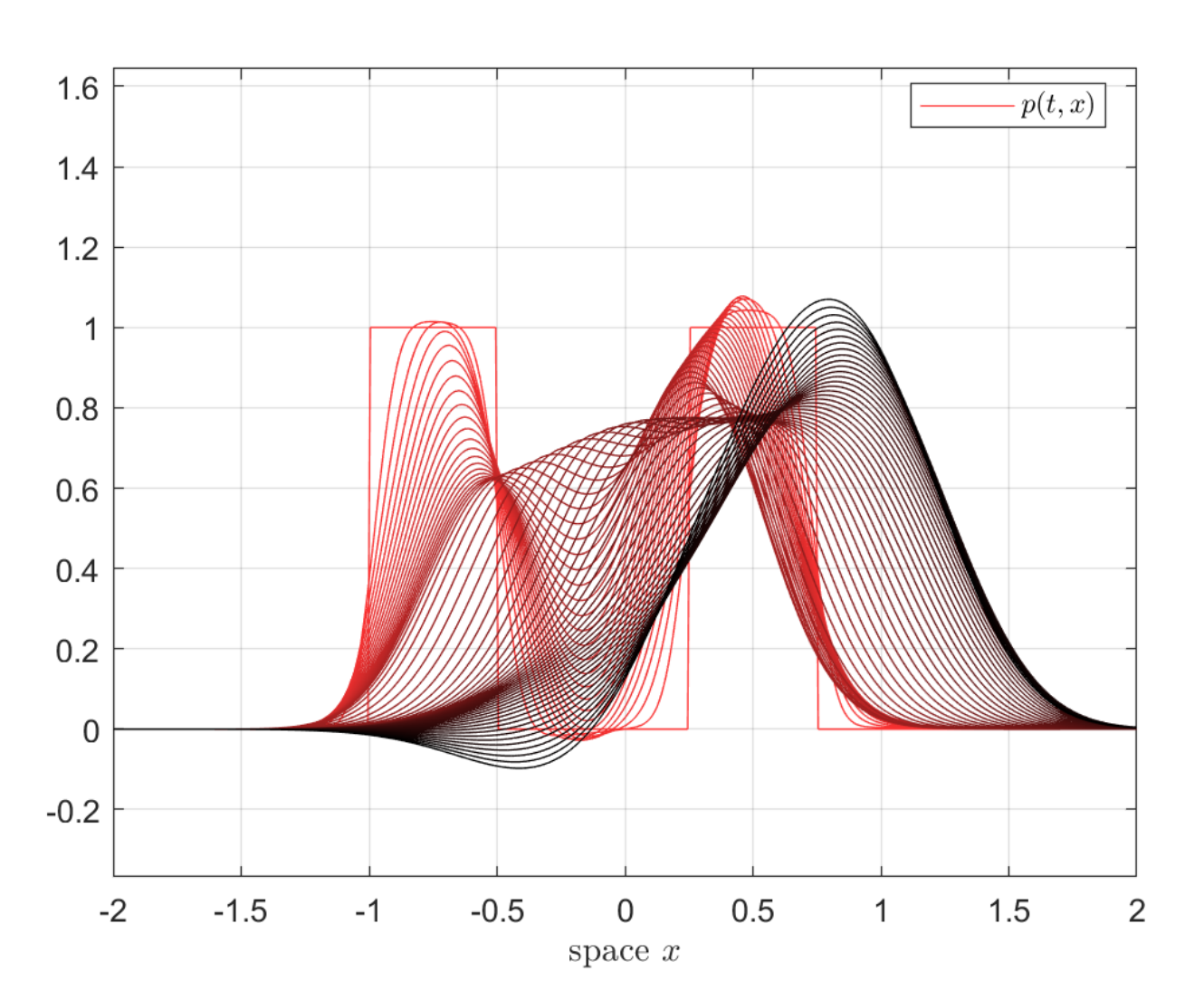}
  \caption{Snapshots of the solution~$p(t_i)$ for $t_i = 0.025\cdot i$, $i=0,\ldots,60$, from red to black.}
  \end{center}
\end{subfigure}\quad
\begin{subfigure}{.48\linewidth}
\begin{center}
  \includegraphics[scale=0.49]{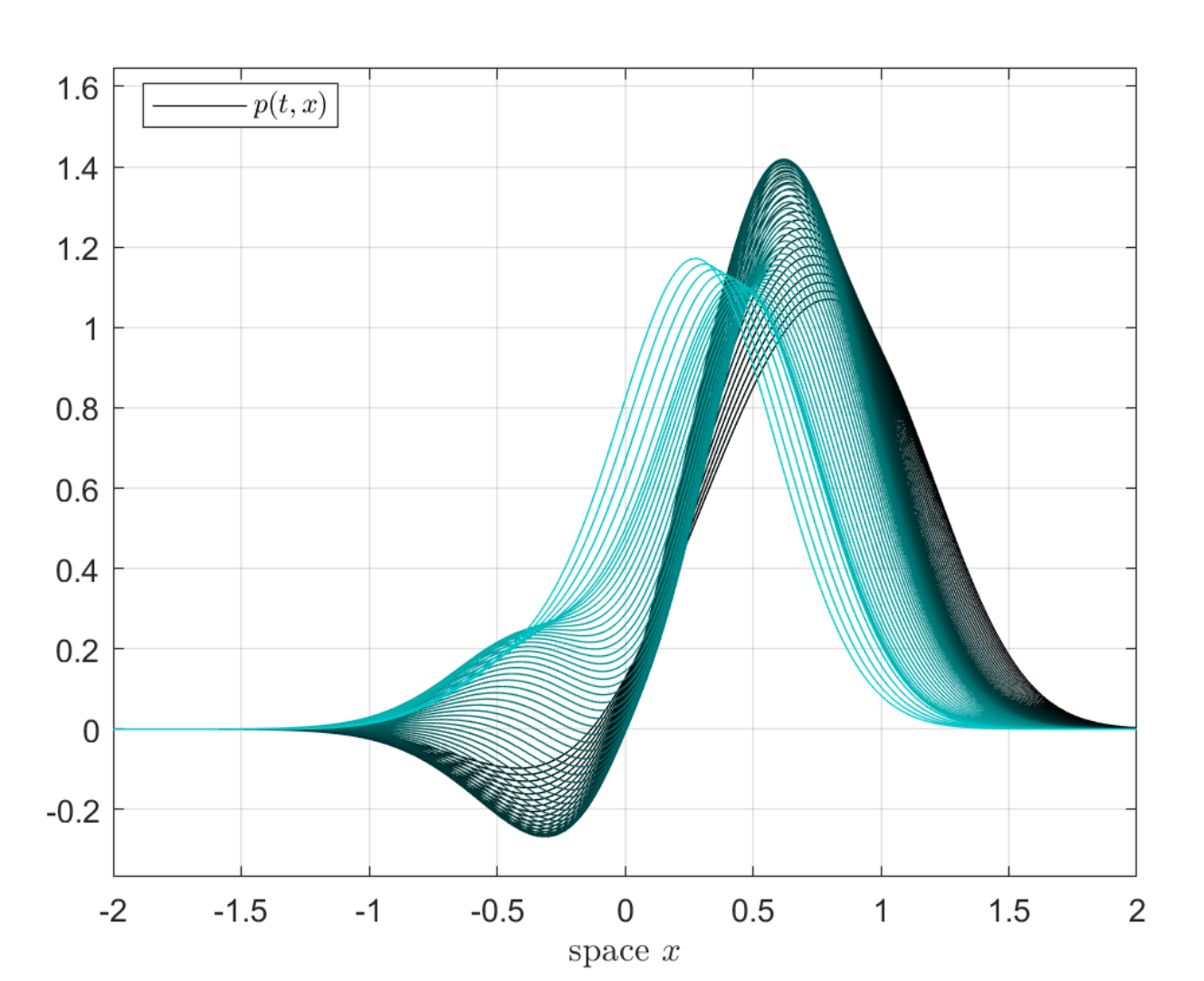}
  \caption{Snapshots of the solution~$p(t_i)$ for $t_i = 1.5 + 0.025\cdot i$, $i=0,\ldots,60$, from black to turquoise.}
  \end{center}
\end{subfigure}
\caption{Simulation of the controller~\eqref{eq:fun-con} applied to~\eqref{eq:FPE-OU-dist} with~\eqref{eq:output} and disturbance~$d$.}
\label{fig:sim-dist}
\vspace{-5mm}
\end{figure}

\begin{figure}[tb!]
\begin{subfigure}{.48\linewidth}
\begin{center}
  \includegraphics[scale=0.49]{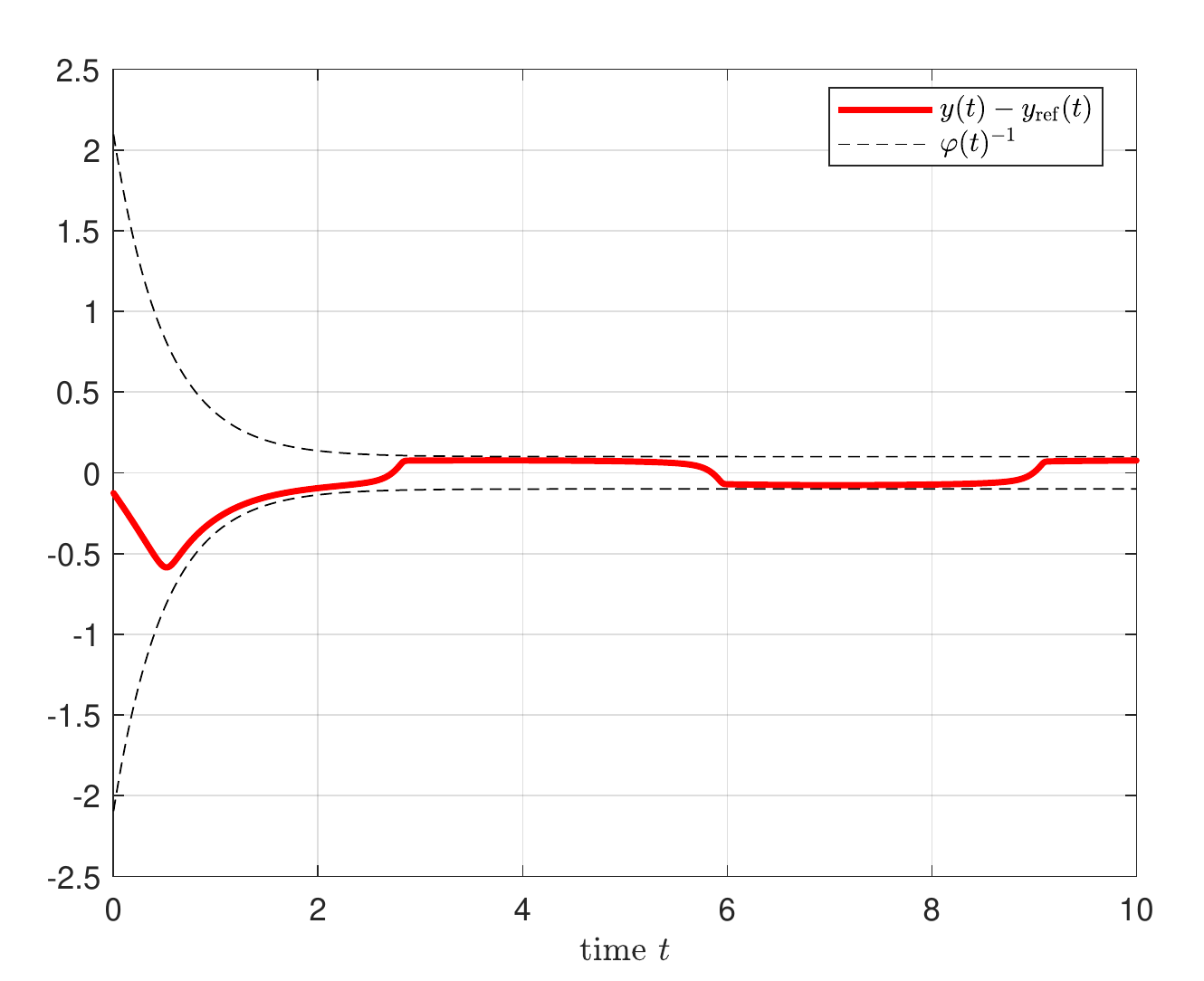}
  \caption{Tracking error and funnel boundary}
  \end{center}
\end{subfigure}\quad
\begin{subfigure}{.48\linewidth}
\begin{center}
  \includegraphics[scale=0.49]{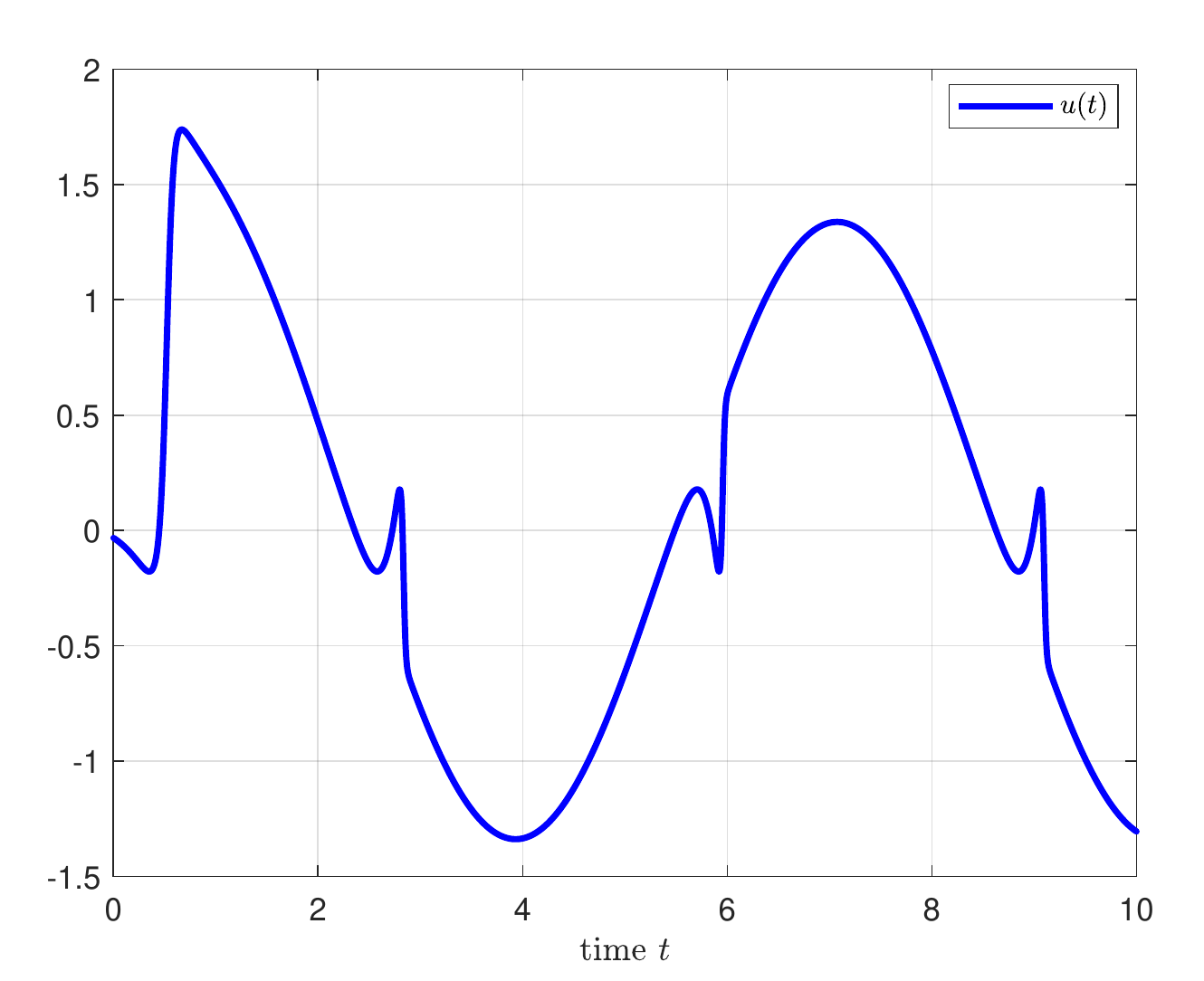}
  \caption{Input function}
  \end{center}
\end{subfigure}
\begin{subfigure}{.48\linewidth}
\begin{center}
  \includegraphics[scale=0.49]{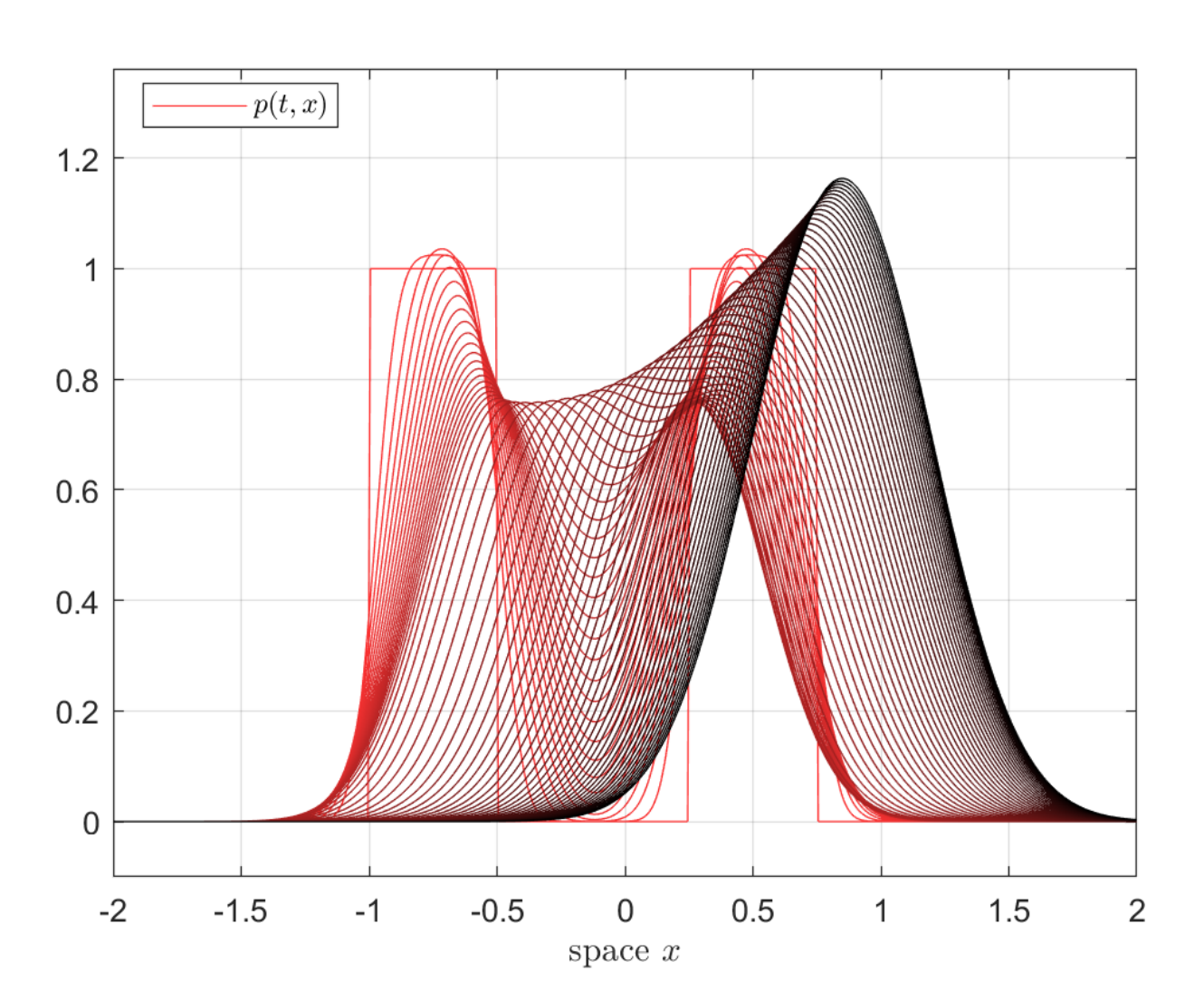}
  \caption{Snapshots of the solution~$p(t_i)$ for $t_i = 0.025\cdot i$, $i=0,\ldots,60$, from red to black.}
  \end{center}
\end{subfigure}\quad
\begin{subfigure}{.48\linewidth}
\begin{center}
  \includegraphics[scale=0.49]{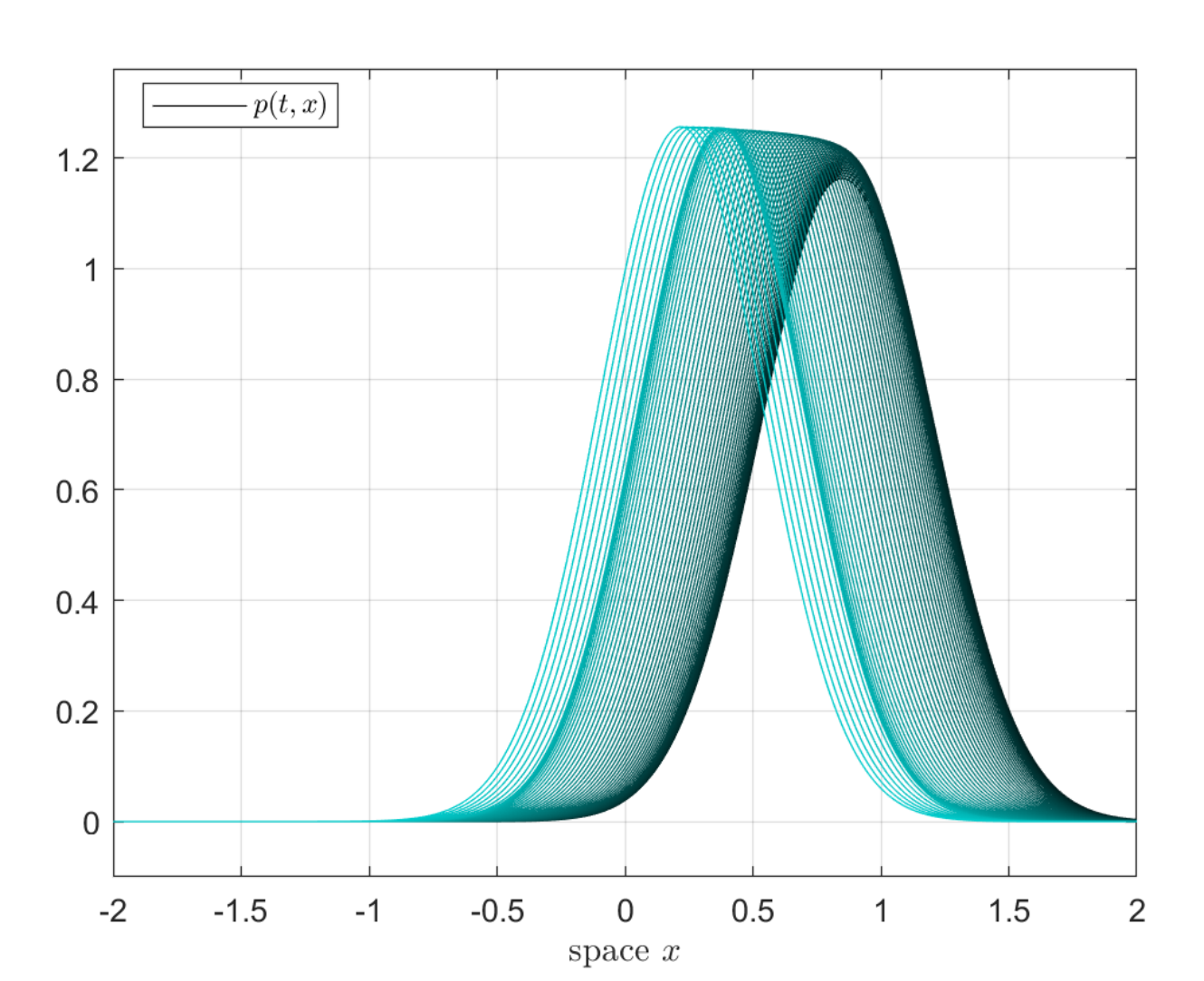}
  \caption{Snapshots of the solution~$p(t_i)$ for $t_i = 1.5 + 0.025\cdot i$, $i=0,\ldots,60$, from black to turquoise.}
  \end{center}
\end{subfigure}
\caption{Simulation of the controller~\eqref{eq:fun-con} applied to~\eqref{eq:FPE-OU-dist} with~\eqref{eq:output}, but without disturbance, i.e., $d=0$.}
\label{fig:sim-undist}
\vspace{-5mm}
\end{figure}

For the simulation the PDE is solved using a finite difference method with a uniform time grid (in $t$) with 10.000 points for the interval $[0,10]$ and a uniform spatial grid (in $x$) with 2.000 points for the interval $[-5,5]$. The simulation has been performed in MATLAB, where in each time step an ODE is solved by using the command \texttt{pdepe} with (artificial) Dirichlet boundary conditions. Relative and absolute tolerance are set to the default values $10^{-3}$ and $10^{-6}$, resp. Fig.~\ref{fig:sim-dist}\,(a) shows the error $e(t) = y(t) - y_{\rm ref}(t)$ between mean value and reference signal and the input values~$u(t)$ generated by the controller are depicted in Fig.~\ref{fig:sim-dist}\,(b). Several snapshots of the solution~$p$, are shown in Fig.~\ref{fig:sim-dist}\,(c) and~(d). It can be seen that, in the presence of disturbances,~$p(t)$ is not a probability density function for $t>0$ in general, since it takes negative values. Nevertheless, the controller guarantees that the error stays within the prescribed funnel boundaries, while the control input shows an acceptable performance.

A simulation of the same configuration, but without disturbance can be seen in Fig.~\ref{fig:sim-undist}. Here, the simulations of the undisturbed equation show that~$p(t)$ is always a probability density and its variance exponentially converges to~$\frac{c}{\Gamma} = 0.2$, as stated in Proposition~\ref{Prop:non-rob-con}. 


\section*{Acknowledgments}

I am indebted to 
Felix L. Schwenninger (U Twente) for several helpful comments and to my PhD student Lukas Lanza (U Paderborn) for helping with the implementation of the simulations.

\bibliographystyle{siam}

\end{document}